\documentclass[%
 aip,
 amsmath,amssymb,
 reprint,%
]{revtex4-1}

\usepackage{graphicx}% Include figure files
\usepackage{dcolumn}% Align table columns on decimal point
\usepackage{bm}% bold math
\usepackage[utf8]{inputenc}
\usepackage[T1]{fontenc}
\usepackage{mathptmx}
\usepackage{etoolbox}
\usepackage{amsmath,amsfonts,amssymb,amsthm}
\usepackage{mathtools}
\usepackage{subfigure}

\newtheorem{theorem}{Theorem}
\newtheorem{prop}[theorem]{Proposition}
\newtheorem{Remark}[theorem]{Remark}

\newcommand{\R}{\mathbb{R}}

\newcommand{\rmO}{\mathrm{O}}

\usepackage{xcolor}
\newcommand{\com}[1]{}

\makeatletter
\def\@email#1#2{%
 \endgroup
 \patchcmd{\titleblock@produce}
  {\frontmatter@RRAPformat}
  {\frontmatter@RRAPformat{\produce@RRAP{*#1\href{mailto:#2}{#2}}}\frontmatter@RRAPformat}
  {}{}
}%
\makeatother
\begin{document}

\preprint{AIP/123-QED}

\title[Linearly-perturbed May--Leonard model]{Dynamics of a linearly-perturbed May--Leonard competition model}
\author{Gabriela Jaramillo}
\affiliation{Department of Mathematics, University of Houston, Houston, TX, 77204, USA}
\author{Lidia Mrad}
\affiliation{Department of Mathematics and Statistics, Mount Holyoke College,  South Hadley, MA, 01075, USA}
\author{Tracy L.\ Stepien}
\affiliation{Department of Mathematics, University of Florida, Gainesville, FL, 32611, USA}

\email{gabriela@math.uh.edu ; lmrad@mtholyoke.edu ; tstepien@ufl.edu}
%\homepage{http://www.Second.institution.edu/~Charlie.Author.}

\date{\today}% It is always \today, today,

\begin{abstract}
The May--Leonard model was introduced to examine the behavior of three competing populations where rich dynamics, such as limit cycles and nonperiodic cyclic solutions, arise. In this work, we perturb the system by adding the capability of global mutations, allowing one species to evolve to the other two in a linear manner. We find that for small mutation rates the perturbed system not only retains some of the dynamics seen in the classical model, such as the three-species equal-population equilibrium bifurcating to a limit cycle, but also exhibits new behavior. For instance, we capture curves of fold bifurcations where pairs of equilibria emerge and then coalesce. As a result, we uncover parameter regimes with new types of stable fixed points that are distinct from the single- and dual-population equilibria characteristic of the original model. On the contrary, the linearly-perturbed system fails to maintain heteroclinic connections that exist in the original system. In short, a linear perturbation proves to be significant enough to substantially influence the dynamics, even with small mutation rates.
\end{abstract}

\maketitle

\begin{quotation}
Almost 50 years ago, May and Leonard~\cite{may-leonard} introduced an extension of the classical Lotka--Volterra nonlinear system to examine the long-term dynamics of three competing populations. In their work, they found that solutions exhibit three distinct behaviors  depending on the parameter values chosen, with the system approaching either a stable fixed point, a periodic orbit or, even more interestingly, what is now known to be a heteroclinic cycle. In the latter case, the observed trajectories are characterized by nonperiodic oscillations of bounded amplitude but ever increasing cycle time. 
Here, we establish and study an extended May--Leonard model by including a linear perturbation that represents the ability of each species to adopt a competing strategy. We find that incorporating the linear perturbation increases the number of physically-relevant equilibrium states for certain parameter values. In addition, we also find that the region in parameter space where periodic orbits exist is much larger than in the case of the original May--Leonard equations, and that the system no longer exhibits nonperiodic cyclic solutions. Therefore, allowing for a small linear mutation term representing global mutations foments coexistence of different species. In biological terms, this would imply that equipping populations with the possibility of switching from one strategy to another with
a small transition or mutation rate can favor biodiversity.
\end{quotation}

\section{Introduction} \label{sec:introduction}
During the last half a century, work on the May--Leonard model~\cite{may-leonard}, a population dynamics model of three competing species, and its variations has led to a variety of results. In particular, Schuster et al.~\cite{schuster-etal} described  the $\omega$-limit set of the original model and proved the existence of a
heteroclinic cycle, while Tang et al.~\cite{tang-etal} constructed a Lyapunov function to find the basin of attraction. It was also determined by Gaunersdorfer~\cite{gaunersdorfer} that the time averages of the trajectories tending to the heteroclinic orbits in the model do not converge but spiral to the boundary of a polygon. Approximate analytic solutions to the system were also derived by Phillipson et al.~\cite{phillipson-etal} and conditions under which the system is integrable have been studied by Leach and Miritzis~\cite{leach-miritzis}, Llibre and Valls~\cite{llibre-valls}, and Bl\'{e} et al.~\cite{ble-etal}. 

Extensions of the May--Leonard model~\cite{may-leonard} have included incorporating asymmetric competitive effects in order to determine conditions for existence and stability of limit cycles and nonperiodic oscillations, as well as existence of first integrals of the Darboux type (Schuster et al.~\cite{schuster-etal}, Chi et al.~\cite{chi-etal}, Wolkowicz~\cite{wolkowicz}, Antonov et al.~\cite{antonov-etal-2016,antonov-etal-2019}). Instead of requiring equal intrinsic growth rates for each competing population as in the May--Leonard model~\cite{may-leonard}, existence of Hopf bifurcations and the stability of steady states were studied under the assumption of unequal intrinsic growth rates (Coste et al.~\cite{coste-etal}, Zeeman~\cite{zeeman}, van der Hoff et al.~\cite{vanderhoff-etal}). Park~\cite{park-2021} extended the model to include an external influx and efflux of individuals into each population. Balanced flow among the groups resulted in persistent coexistence of all groups, including cases with oscillatory dynamics, while imbalanced flow resulted in various population survival states. More examples of various general three-species competition models can be found in the review paper by Dobramysl et al.~\cite{dobramysl-review-2018}.

The same year that the May--Leonard model~\cite{may-leonard} was published, Gilpin~\cite{gilpin} considered the effects of adding a constant perturbation to these equations. He found that this constant term allowed for the formation of limit cycles in regions of parameter space where the original model exhibited only nonperiodic oscillations. Other types of perturbations have not been considered until more recently. For example, in 2014, Zhao and Cen~\cite{zhao-cen} showed that adding small quadratic perturbations to the model results in exactly one or two limit cycles bifurcating from the periodic orbits of the May--Leonard system. Other perturbations recently studied have been periodic in nature. In particular, it has been found that periodically forcing the May--Leonard system results in the existence of strange attractors (Rodrigues~\cite{rodrigues}). Additionally, periodic, quasiperiodic, and chaotic solutions have been shown to exist under different parameter conditions for small periodic perturbations to the asymmetric May--Leonard model (Afraimovich et al.~\cite{afraimovich-etal}) and time-periodic perturbations to a general 3-D competitive Lotka-Volterra model, of which the May--Leonard model is a subcase (Chen et al.~\cite{chen-etal}).

In this work, we add a linear perturbation to the symmetric May--Leonard model. This linear perturbation models mutations among the competing populations, whereby individuals in one class are able to mutate into another class. The first to examine these types of perturbations in a rock--paper--scissors model with replicator-mutator equations was Mobilia~\cite{mobilia-2010} about a decade ago, followed by Toupo and Strogatz~\cite{toupo-strogatz}, among others (Yang et al.~\cite{yang-etal-2017}, Park~\cite{park-2018}, Hu et al.~\cite{hu-etal-2019}, Mittal et al.~\cite{mittal-etal-2020}, Kabir and Tanimoto~\cite{kabir-tanimoto-2021}, Mukhopadhyay et al.~\cite{mukhopadhyay-etal-2021}). Both the replicator equations and the May--Leonard model have a similar structure, with the main difference being that in the former case the unknowns represent fractions of a fixed population,  while in the latter case the total population is not assumed to be a fixed number a priori. 

As was the case in the May--Leonard equations, depending on the parameter values chosen, the trajectories of the replicator equations exhibit three types of long-term behavior. Solutions can either approach the equal-population stable fixed point, a heteroclinic cycle or, in contrast to the May--Leonard model, one of the infinitely many neutrally stable cycles that fill the state space. The effect of adding global mutations to this system, where each species can mutate to any of the other two with the same rate, is the loss of the saddle fixed points that form the heteroclinic cycle, as well as  the emergence of a stable limit cycle from a supercritical Hopf bifurcation for certain parameter values (Mobilia~\cite{mobilia-2010}). In contrast, we find that adding  to the May--Leonard system  a linear perturbation modeling global mutations increases the number of physically-relevant steady states for certain parameter values, and consequently changes the ensuing dynamics. We summarize our findings, for small mutation rates, below:
\begin{itemize}
    \item As expected, the perturbation changes the nature of some steady states. We recover the trivial and equal-population equilibria; however, we also find a richer variety of fixed points that we view as perturbations of the single- and dual-population equilibria found in the May--Leonard model. 
    \item \com{
    As in the original equations,
    the linearly-perturbed system
    exhibits periodic solutions. However, while for the 
    May--Leonard model these trajectories are only present along a line in parameter space, in the case of the linearly-perturbed system, they can be observed for a wider set of parameter values.}

    \item In contrast to the May--Leonard system, the numerical results we present suggest that the linearly-perturbed May--Leonard model does not  possess heteroclinic cycles. \com{This means that the system no longer exhibits nonperiodic cyclic solutions of bounded amplitude but increasing cycle time.}
  \end{itemize}

Practical applications of the May--Leonard model~\cite{may-leonard} in the existing scientific literature are limited in number; \com{the model is nonetheless used, in one instance in the literature,} to calculate the cropping quotas for three competing herbivore species (Fay and Greeff~\cite{fay-greeff}).
Such limited applicability stems, in part, from the assumption that populations follow a cyclic dominance competition pattern, which can be restrictive. 
\com{For example, in the parameter regime that predicts nonperiodic oscillations in the May--Leonard model, no single species dominates, as the other two species will always rebound in size. However, after just a few cycles, the value of the two smaller populations will decrease below unity. As pointed out by May and Leonard in their paper, this behavior is idealized, and if we were to consider a real population, these species would instead go extinct (May and Leonard~\cite{may-leonard}).}

\com{On the other hand,}
it is in \com{the cyclic dominance competition pattern} aspect that the model's equations resemble the replicator equations used to model evolutionary games. Indeed, evolutionary games are widely used in theoretical biology to study interactions between species which follow a cyclic dominance pattern (Mobilia~\cite{mobilia-2010}, Cz\'{a}r\'{a}n et al.~\cite{czaran2002}, Kerr et al.~\cite{kerr2002}, Szolnoki et al.~\cite{szolnoki2014}, Hofbauer and Sigmund~\cite{hofbauer1998}). Although this form of competition seems to be rare in nature, there are a few examples where this behavior occurs. These include the mating strategies of side-blotched lizards (Sinervo and Lively~\cite{sinervo1996}, Zamudio and Sinervo~\cite{zamudio2000}), and the interactions between three different strains of E. coli (Kerr et al.~\cite{kerr2002}). In this context, mutation can be seen as the ability of a population to change its competing strategy. Previous work in this area by Toupo and Strogatz~\cite{toupo-strogatz} and Mobilia~\cite{mobilia-2010} has shown that global mutations result in the emergence of a limit cycle. It is perhaps then not surprising that the numerical and analytical results we present here lead to the same conclusion.

{\bf Outline:} This paper is organized as follows. We first review the fixed points of the May--Leonard model~\cite{may-leonard} and their corresponding stability in Section~\ref{sec:may-leonard}. We then introduce the linearly-perturbed May--Leonard model in Section~\ref{sec:may-leonard-perturbation}, and explore how this modification alters the dynamics for different parameter regimes in Section~\ref{sec:fixedpoints}.
In particular, we find that the system fails to maintain the heteroclinic connections that exist in the original model, justifying the lack of nonperiodic oscillations we observe in simulations. Finally, we summarize our findings in Section~\ref{sec:discussion}, where we further comment on the effects of adding the linear perturbation to the model.

%%%%%%%%%%%%%%%%%%%%%%%%%%%%%%%%%%%%%%%%%%%%%%%%%%%%%%%%%%%%%%%%%%%
\section{May--Leonard Model} \label{sec:may-leonard}

May and Leonard~\cite{may-leonard} extended the classic Lotka--Volterra equations for two competitors to a system of three competitors, $m_1(t)$, $m_2(t)$, and $m_3(t)$, described by equations of the general form
\begin{equation} \label{e:ML_general}
    \frac{d m_i(t)}{dt} = r_i m_i(t) \left( 1-\sum_{j=1}^3 \alpha_{ij} m_j(t) \right) , \qquad i=1,2,3.
\end{equation}
Under symmetry assumptions that the intrinsic growth rates are equal, $r:=r_1=r_2=r_3$, and that the competitors affect each other in a cyclic manner such that $\alpha:=\alpha_{12}=\alpha_{23}=\alpha_{31}$ and $\beta:=\alpha_{21}=\alpha_{13}=\alpha_{32}$, along with a rescaling of the populations $m_i$ and time $t$ such that $\alpha_{ii}=1$ and $r=1$, the May--Leonard model~\cite{may-leonard} becomes
\begin{subequations} \label{e:ML}
\begin{align}
    \frac{dm_1}{dt} &= m_1 \Big(1-m_1-\alpha m_2 - \beta m_3 \Big) , \\
    \frac{dm_2}{dt} &= m_2 \Big(1-\beta m_1- m_2 - \alpha m_3 \Big) , \\
    \frac{dm_3}{dt} &= m_3 \Big(1-\alpha m_1-\beta m_2 - m_3 \Big) .
\end{align}
\end{subequations}

Solutions to \eqref{e:ML} tend to one of the system's 8 fixed points, a limit cycle, or a nonperiodic oscillation of bounded amplitude but increasing cycle time.

%%%%%%%%%%%%%%%%%%%%%%%%%%%%%%%%%%%%%%%%%%%%%%%%%%%%%%%%%%%%%%%%%%%
\subsection{Fixed Points and Stability}

The May--Leonard model~\eqref{e:ML} possesses 5 distinct nonnegative fixed points,
\begin{subequations} \label{e:fixedpointsML}
\begin{gather}
     e_0 =(0,0,0),  \label{e:fixedpointsML_zero} \\[0.5em]
     e_1 = (1,0,0), \qquad e_2 = (0,1,0), \qquad e_3 = (0,0,1) ,  \label{e:fixedpointsML_singlepop} \\
     e_c = \left( \frac{1}{1+\alpha+\beta},\,\, \frac{1}{1+\alpha+\beta},\,\, \frac{1}{1+\alpha+\beta} \right),  \label{e:fixedpointsML_equalpop}
\end{gather}
\end{subequations}
known to exist for all values of $\alpha, \beta >0$, as well as 3 dual-population fixed points
\begin{subequations} \label{e:fixedpointsML_dualpop}
\begin{align}
    f_1 &= \left(0,\,\, \tfrac{1-\alpha}{1-\alpha \beta},\,\, \tfrac{1-\beta}{1-\alpha \beta} \right) , \\
    f_2 &= \left( \tfrac{1-\beta}{1-\alpha \beta},\,\,  0,\,\,  \tfrac{1-\alpha}{1-\alpha \beta} \right) , \\
    f_3 &= \left(\tfrac{1-\alpha}{1-\alpha \beta},\,\, \tfrac{1-\beta}{1-\alpha \beta},\,\, 0\right) ,
\end{align}
\end{subequations}
for which positivity, and thus their physical relevance, depends on the values of $\alpha$ and $\beta$. For example, for these fixed points to exist, we require $\alpha \beta \neq 1$.

The stability of these fixed points as a function of the two parameters $\alpha$ and $\beta$ is studied in depth in May and Leonard~\cite{may-leonard}. Their results are summarized in the stability diagram in Fig.~\ref{f:may-leonard-stability}, which we will also describe here.

\begin{figure}[h]
   \centering
   \includegraphics[width=0.9\columnwidth]{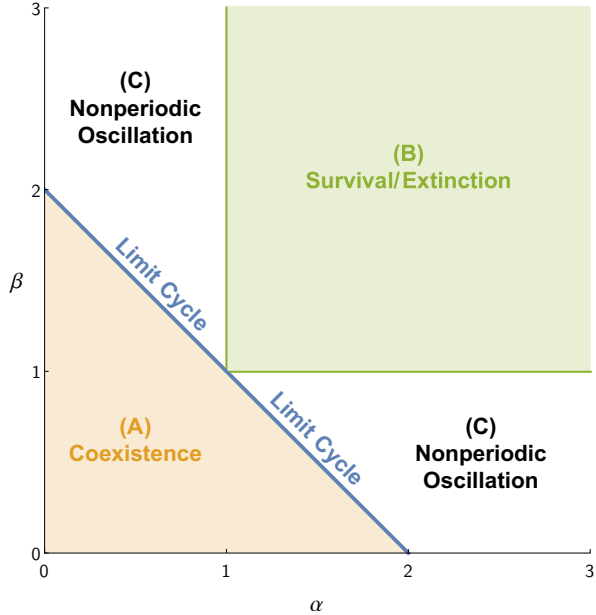} 
   \caption{Stability diagram of fixed points, limit cycles, and nonperiodic oscillations of the May--Leonard model~\eqref{e:ML}. In Region A, only the equal-population fixed point $e_c$~\com{\eqref{e:fixedpointsML_equalpop}} is stable. In Region B, which is bounded by the lines $\alpha=1$ and $\beta=1$, the single-population fixed points $e_1$, $e_2$, and $e_3$~\com{\eqref{e:fixedpointsML_singlepop}} are stable. In Region C, nonperiodic oscillations exist. Along the line $\alpha+\beta=2$, limit cycles exist.}
   \label{f:may-leonard-stability}
\end{figure}

The fixed point at the origin, $e_0$, is always unstable. In Region A, the only stable fixed point is the equal-population fixed point, $e_c$. 

In Region B, the situation is reversed and all single-population fixed points, $e_1$, $e_2$, and $e_3$, are stable, while the fixed point $e_c$ is now unstable. In this region, the long term dynamics of the system depend on the initial conditions, and thus the system approaches one of the fixed points, $e_1$, $e_2$, or $e_3$, according to its initial configuration.

In Region C, the system has nonperiodic cyclic solutions that lie on the hyperplane \mbox{$m_1+m_2+m_3=1$}. These solutions approach and then leave each of the single-population fixed points. The time the system spends near each $e_i$ increases as the system evolves, and this loitering behavior follows a logarithmic scale. On the border between Regions A and C, where the parameters satisfy $\alpha + \beta =2$, the system exhibits a limit cycle.

%%%%%%%%%%%%%%%%%%%%%%%%%%%%%%%%%%%%%%%%%%%%%%%%%%%%%%%%%%%%%%%%%%%
\section{May--Leonard Model With Linear Perturbations} \label{sec:may-leonard-perturbation}

We extend the May--Leonard model~\eqref{e:ML_general} to include linear perturbations that are of the same form as the ``global mutations'' in Toupo and Strogatz~\cite{toupo-strogatz}, where each population $m_i$ can mutate into the other two with rate $\mu$. The general form of this linearly-perturbed May--Leonard model, is
\begin{equation} \label{e:MLperturb_general}
	\frac{dm_i}{dt} = r_i m_i \left(1 - \sum_{j=1}^3 \alpha_{ij} m_j \right) + \mu \left( -2 m_i + \sum_{\substack{j = 1 \\ j\neq i}}^3 m_j \right),
\end{equation}
for $i=1,2,3$. Assuming, as in Section~\ref{sec:may-leonard}, equal intrinsic growth rates and that the competitors affect each other in a cyclic manner, along with the same rescaling of populations $m_i$ and time $t$, \eqref{e:MLperturb_general} becomes
\begin{subequations} \label{e:MLperturb}
\begin{align}
	\frac{dm_1}{dt} &= m_1 \Big(1 - m_1 - \alpha m_2 - \beta m_3 \Big) + \mu \Big( -2m_1 + m_2 + m_3 \Big) , \\
	\frac{dm_2}{dt} &= m_2 \Big(1 - \beta m_1 - m_2 - \alpha m_3 \Big) + \mu \Big( m_1 -2 m_2 + m_3 \Big) , \\
	\frac{dm_3}{dt} &= m_3 \Big(1 - \alpha m_1 - \beta m_2 - m_3 \Big) + \mu \Big( m_1 + m_2 - 2 m_3 \Big) . 
\end{align}
\end{subequations}

We assume that the competition parameters $\alpha,\beta>0$ and the mutation parameter $\mu>0$. In the rest of this paper, we study how the stability diagram of the May--Leonard model (Fig.~\ref{f:may-leonard-stability}) changes when the mutation parameter, $\mu$, in the linearly-perturbed model~\eqref{e:MLperturb} is nonzero.

%%%%%%%%%%%%%%%%%%%%%%%%%%%%%%%%%%%%%%%%%%%%%%%%%%%%%%%%%%%%%%%%%%%

\section{Stability Diagram} \label{sec:fixedpoints}

In this section, we use perturbation analysis and the continuation software package AUTO 07 \cite{auto07p} to investigate the effects of the mutation parameter, $\mu$, on the number and stability of nonnegative fixed points in system~\eqref{e:MLperturb}.
 Our results are summarized in Fig.~\ref{f:phasediagram-MLRPS}.

\begin{figure}[h] 
   \centering
   \includegraphics[width=0.9\columnwidth]{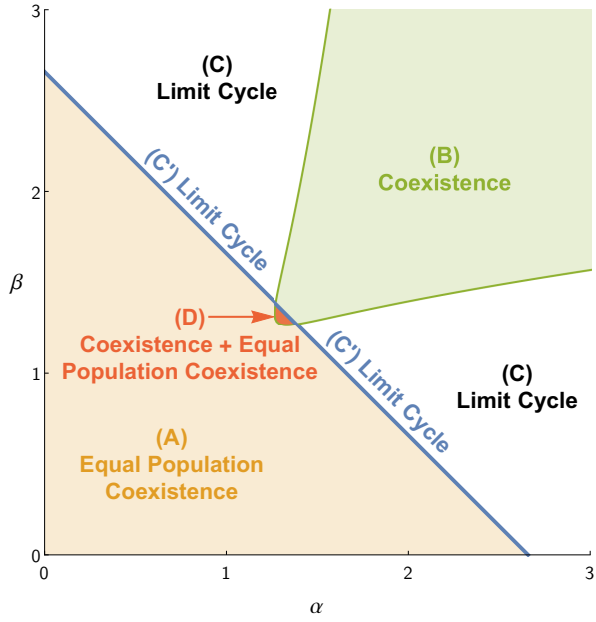}
   \caption{Stability regions of the linearly-perturbed May--Leonard model~\eqref{e:MLperturb} for $\mu =0.03$. 
   In Region A, only the equal-population fixed point $e_c$~\com{\eqref{e:fixedpointsML_equalpop}} is stable. 
   Crossing into Regions B and C, this fixed point loses stability in a Hopf bifurcation giving rise to a limit cycle that exists in Regions C and C$^\prime$.
 In Region B, six new fixed points emerge from a fold bifurcation, three of which are stable while the other three are unstable. 
    Region D differs from Region B only in the fact that the equal-population fixed point \com{$e_c$~\eqref{e:fixedpointsML_equalpop}} is stable in Region D.} \label{f:phasediagram-MLRPS}
\end{figure}

%%%%%%%%%%%%%%%%%%%%%%%%%%%%%%%%%%%%%%%%%%%%%%%%%%%%%%%%%%%%%%%%%%%

\subsection{Fixed Points}

We first focus on how the steady states of the linearly-perturbed May--Leonard model~\eqref{e:MLperturb} change as $\mu$ increases. A short computation shows that the fixed point at the origin, $e_0$~\eqref{e:fixedpointsML_zero}, 
and the equal-population steady state, $e_c$~\eqref{e:fixedpointsML_equalpop}, persist for all values of $\mu>0$. Though the expression for $e_c$~\eqref{e:fixedpointsML_equalpop} depends only on $\alpha$ and $\beta$, we find that  its stability depends in a nontrivial way on the parameter  $\mu$. Indeed, in Section \ref{ss:regionA}  we show that this fixed point  undergoes a Hopf bifurcation at a critical value, $\mu_c = \mu_c(\alpha, \beta)$.

On the other hand, when $\mu>0$, we no longer find single- and dual-population fixed points
  corresponding to \eqref{e:fixedpointsML_singlepop} and \eqref{e:fixedpointsML_dualpop} of the May--Leonard model~\eqref{e:ML}.
Instead, depending on the parameters $\alpha$, $\beta$, and $\mu$, the system \com{ might exhibit} six triple-population equilibria. Due to the symmetries of the system, these steady states can be split into two families, 
where members within a family can be mapped to each other by permuting their components.
We also find that these six fixed points disappear in a fold bifurcation as the value of $\mu$ is increased.
In particular, fixing the value of $\beta$, one can numerically compute two sets of curves in the $\mu$--$\alpha$ plane where this bifurcation occurs (Fig.~\ref{fig:folds_mu}). 

\begin{figure}[h] 
   \centering
   \subfigure[$\beta = 0.5$ \label{fig:beta05}]{\includegraphics[width=0.75\columnwidth]{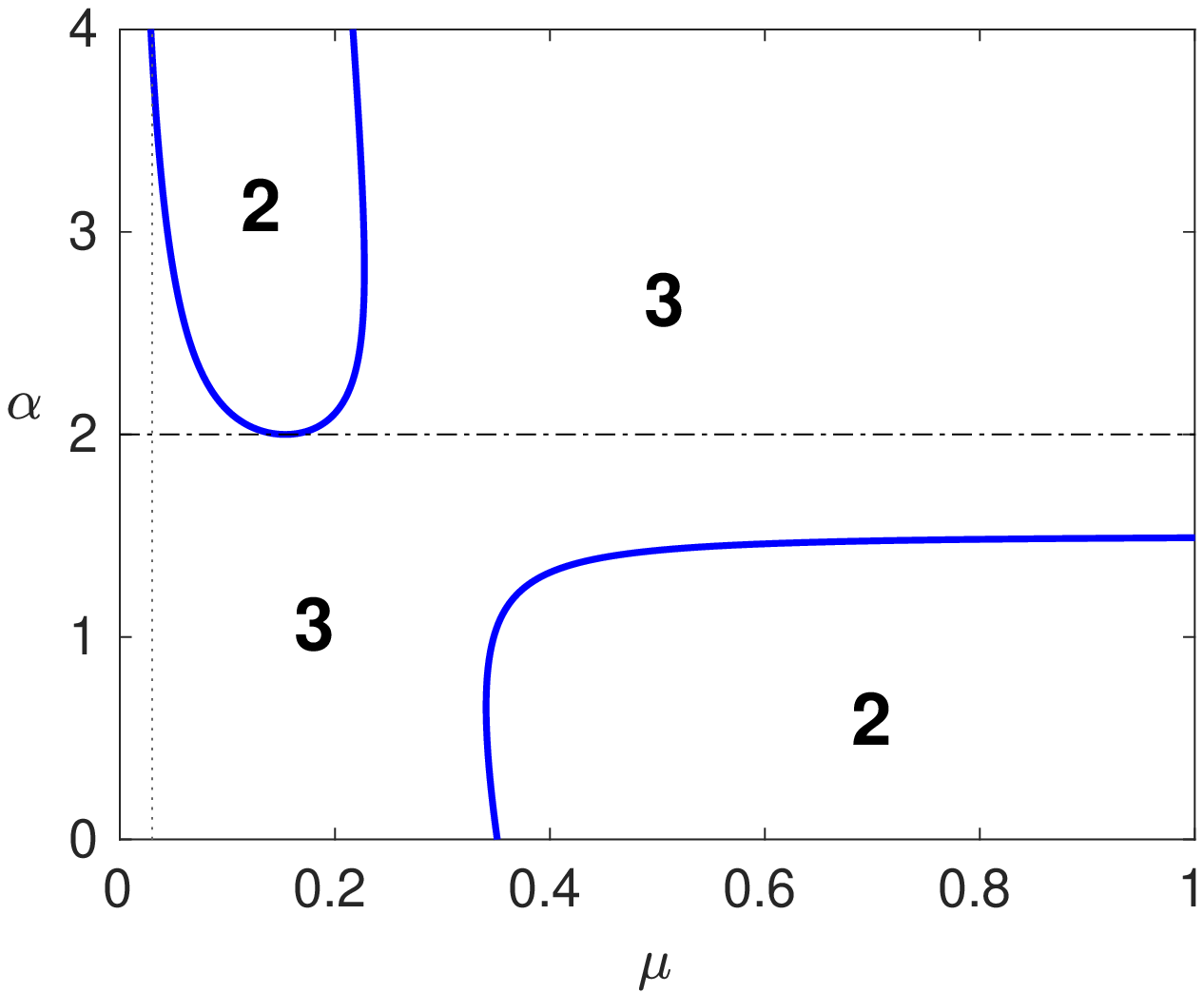}}
   \subfigure[$\beta =1.5$ \label{fig:beta15}]{\includegraphics[width=0.75\columnwidth]{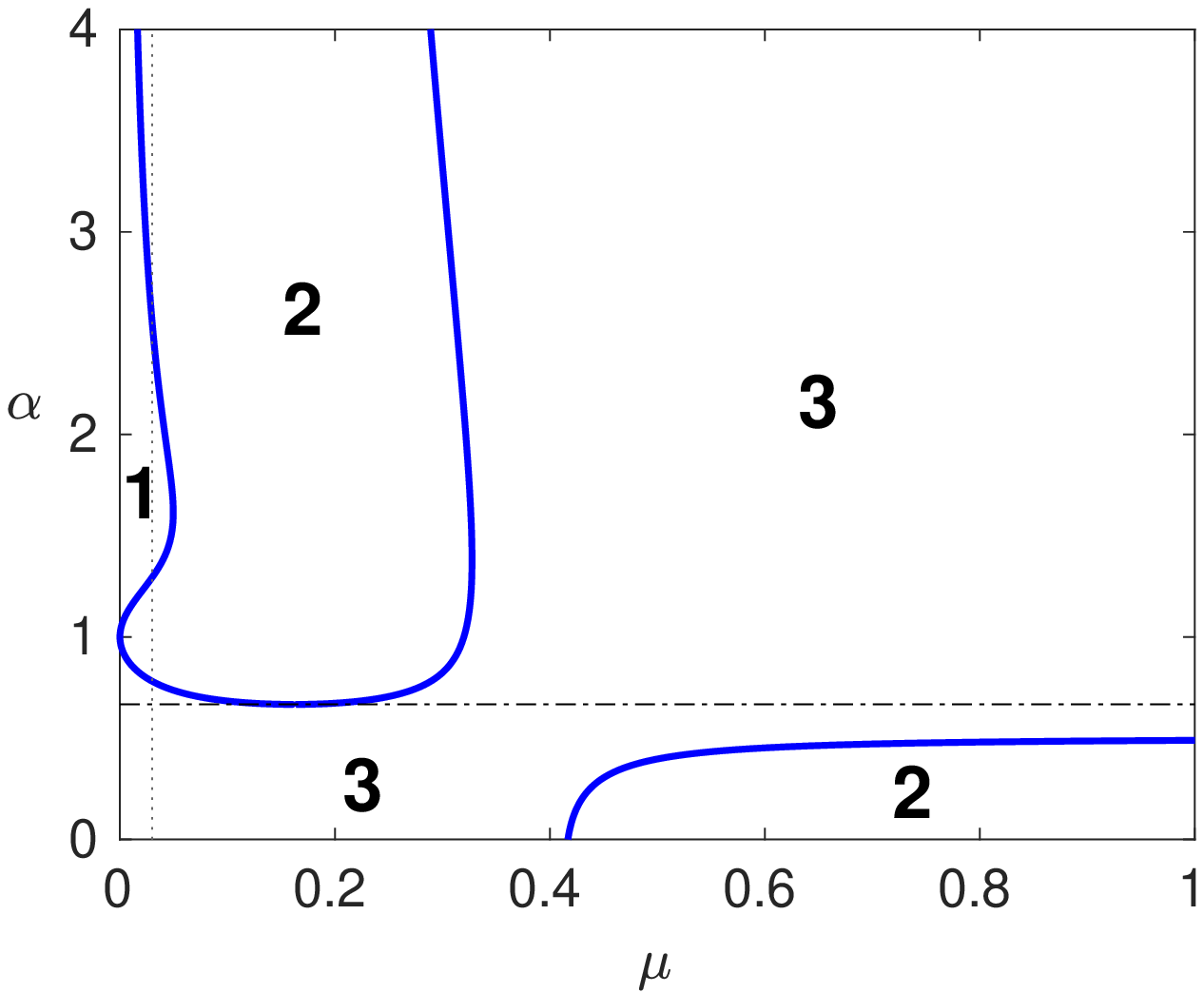} }
   \caption{Curves of fold bifurcations split the $\mu$--$\alpha$ plane into distinct regions. In Region~1, \com{the components of the} triple-population fixed points \com{are all positive}. In Region~3, the triple-population \com{fixed points} have at least one negative component. In Region~2,  no  triple-population fixed points exist. The horizontal dash-dotted line represents the condition $\alpha\beta =1$, where the triple-population fixed point has an unbounded component (see Appendix~\ref{sec:appendix}). The vertical dotted line represents $\mu=0.03$, which corresponds with the value used to produce the bifurcation curves in Fig.~\ref{fig:folds}.}
   \label{fig:folds_mu}
\end{figure}

We distinguish whether the system supports triple-population fixed points with only positive components or not. In particular, we find that for small values of $\beta$, the only positive steady state 
is the equal-population equilibrium $e_c$~\eqref{e:fixedpointsML_equalpop}, whereas for larger values of $\beta$, positive triple-population fixed points exist. This can be seen in Fig.~\ref{fig:beta05} and \ref{fig:beta15}, where we take $\beta = 0.5$ and $\beta=1.5$, respectively.  In both figures,  regions labeled with the number 2 correspond to parameter values where no triple-population equilibria exist. Similarly, regions labeled as 3 correspond to values of $\alpha$ and $\mu$ where these fixed points appear but have at least one negative component. In the regions labeled as 3, we also observe that for certain parameter values, which we plot as a dotted line, these fixed points have an unbounded component. Finally, the region labeled as 1, which is only present in Fig.~\ref{fig:beta15}, consists of those parameter values  where triple-population fixed points exist and have only positive components. To understand the emergence of the  triple-population equilibria in this region, we fix the value of $\mu$ to be a very small number and track these steady states as the other two parameters are varied.

As we look for equilibrium points in the $\alpha$--$\beta$ plane (for small $\mu$), we discover a series of fold bifurcations that define regions of existence of triple-population equilibria. Specifically, when  $\alpha$ and $\beta$ are small, we find that all six fixed points have at least one negative component and are therefore not physically relevant. As the values of $\alpha$ and $\beta$ increase, these two families of equilibria collide and disappear at a curve of fold bifurcations, which is depicted in Fig.~\ref{fig:folds} by the left-most dashed curve (labeled as 1). The only fixed points that exist here are the origin and the equal-population equilibria. As $\alpha$ and $\beta$ increase further, a new set of six equilibria re-emerge, this time with positive components. The second fold bifurcation where this occurs is depicted in Fig.~\ref{fig:folds} by the right-most dashed curve (labeled as 2). As can be seen in Fig.~\ref{fig:folds}, this disappearance and re-emergence can happen more than once for specific $\alpha$ or $\beta$ values. For example, fix $\alpha$ at $\alpha=1.5$. For small $\beta$ values, there are no physically relevant equilibrium points, other than the origin and the equal-population ones. As $\beta$ increases to around 1.29, a set of six triple-population points emerge. When $\beta$ reaches 2.55, these points coalesce and disappear.

\begin{figure}[h] 
   \centering
   \includegraphics[width=0.9\columnwidth]{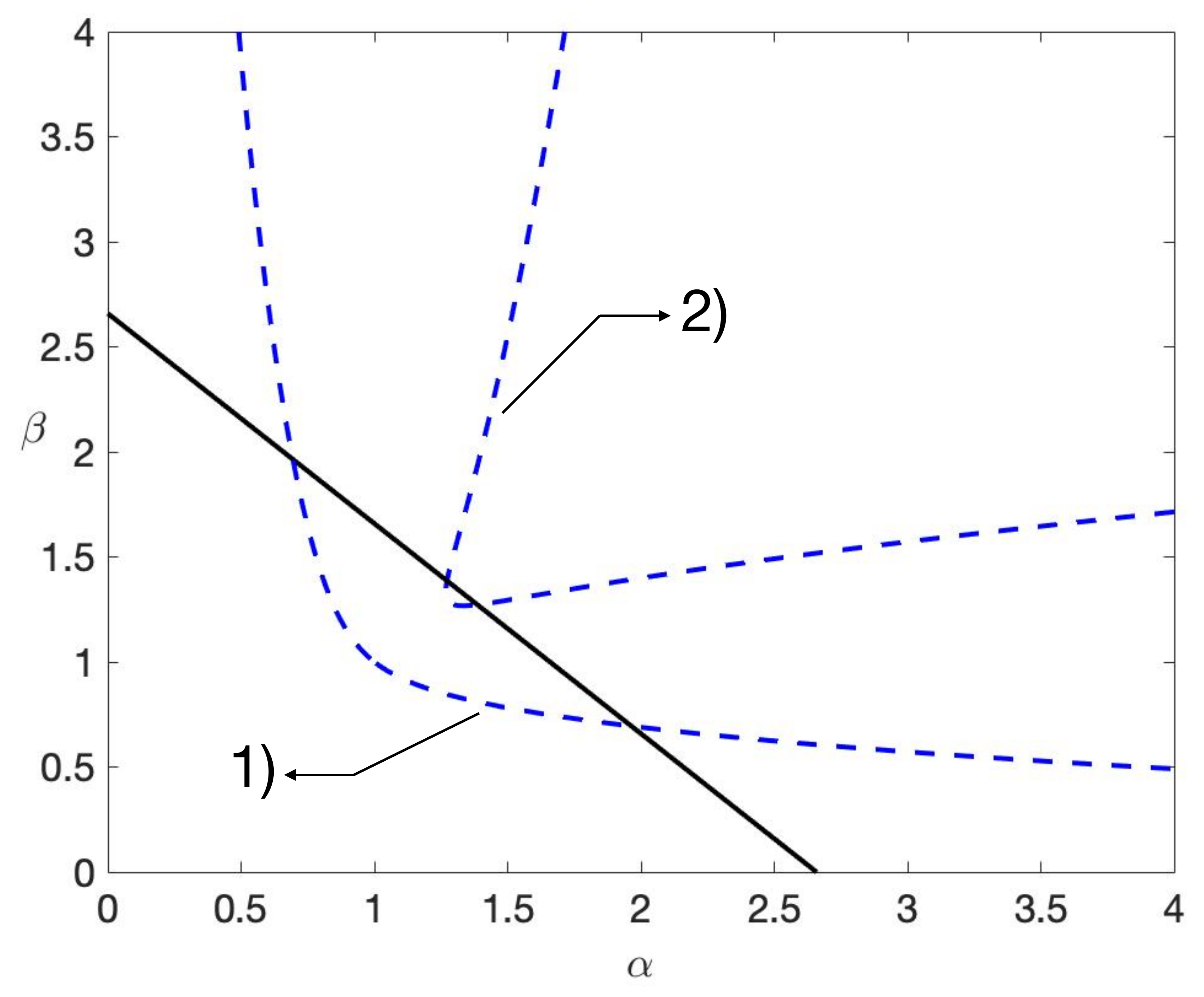} 
   \caption{Bifurcation curves for the linearly-perturbed May--Leonard model~\eqref{e:MLperturb} 
   with \mbox{$\mu =0.03$}. Curves were computed using AUTO 07~\cite{auto07p}. The solid line represents the locus of Hopf bifurcations for
   the fixed point $e_c$~\eqref{e:fixedpointsML_equalpop}. The dashed curves \com{labeled 1) and 2)} represent the locus of fold bifurcations.
    In between \com{these curves}, the only equilibria that exist are $e_0$~\com{\eqref{e:fixedpointsML_zero}} and $e_c$~\com{\eqref{e:fixedpointsML_equalpop}}.
 To the left of curve 1), we find that six additional equilibria emerge, all of which have at least one negative component, while
  to the right of curve 2), we find a different set of six additional positive equilibria. }
   \label{fig:folds}
\end{figure}

 Going back to Fig.~\ref{fig:folds_mu}, we want to relate Regions 1, 2, and 3 in this diagram with the sections in Fig.~\ref{fig:folds} which are separated by dashed curves. We observe that along the vertical dotted line representing $\mu=0.03$ in Fig.~\ref{fig:beta15}, Region 3 corresponds with the section to the left of the dashed curve labeled 1) in Fig.~\ref{fig:folds}, where the triple-population fixed points have at least one negative component. As we continue up the line $\mu=0.03$ into Region 2 in Fig.~\ref{fig:beta15}, we move into the section in between the dashed curves labeled 1) and 2) in Fig.~\ref{fig:folds}, where only the equal-population, $e_c$~\eqref{e:fixedpointsML_equalpop}, and zero-population fixed point, $e_0$~\eqref{e:fixedpointsML_zero}, exist. Going back to Fig.~\ref{fig:beta15} 
and crossing into Region 1,  we move into the section to the right of the dashed curve labeled 2) in Fig.~\ref{fig:folds}, 
where the triple-population fixed points all have positive components. A similar description can be made for the case when $\beta =0.5$ (Fig.~\ref{fig:beta05}).

 Fig.~\ref{fig:folds} also shows how the locus of Hopf and fold bifurcations divide
 the $\alpha$--$\beta$ plane. The right-most curve of fold bifurcations (labeled as 2) together with the Hopf line
  create the four stability regions depicted in Fig.~\ref{f:phasediagram-MLRPS}. 
  We find that in Regions A and C, the only nonnegative equilibria are $e_0$~\eqref{e:fixedpointsML_zero} and $e_c$~\eqref{e:fixedpointsML_equalpop},
 while in Regions B and D, we have an additional six positive steady states.
  In the following subsections, we describe the dynamics of the system in each of these four regions, as well as along the Hopf bifurcation line.

%%%%%%%%%%%%%%%%%%%%%%%%%%%%%%%%%%%%%%%%%%%%%%%%%%%%%%%%%%%%%%%%%%%
\subsection{Region A}\label{ss:regionA}

In this subsection, we investigate the stability of the equal-population fixed point, $e_c$~\eqref{e:fixedpointsML_equalpop}, in Region A. We find that for parameter values of
$\mu \geq 1/6$, this fixed point is stable for all positive values of $\alpha$ and $\beta$, while for values of $\mu< 1/6$, the equilibrium point loses its stability through a Hopf bifurcation.
In the latter case, we also obtain an expression, $\beta_c = \beta_c(\alpha, \mu)$, for the location in the $\alpha$--$\beta$ plane  where this transition occurs. Given a fixed value $\mu^*$, the line $\beta_c = \beta_c(\alpha, \mu^*)$ then 
represents the stability boundary for the equal-population fixed point. In fact, Region A is defined to be the region below this line, for which a precise expression is obtained later in equation \eqref{e:HopfBifurcationLine}. The following \com{theorem} summarizes what we know about Region A.
\begin{theorem}
For $\alpha,\:\beta>0$, the fixed point, $e_c$~\eqref{e:fixedpointsML_equalpop}, is always stable for $\mu \geq 1/6$. But for $\mu < 1/6$, it is only stable when $\alpha+\beta < (6\mu+2)/(1-6\mu).$
\end{theorem}

\begin{proof}
We start with the Jacobian of system \eqref{e:MLperturb} evaluated at the equilibrium point $e_c$~\eqref{e:fixedpointsML_equalpop},
\begin{equation}
\label{e:equaltriplefp_jacobian}    
J  = \frac{-1}{1+\alpha+\beta}\begin{bmatrix} 
1 & \alpha & \beta\\
\beta & 1 & \alpha\\
\alpha& \beta &1 
\end{bmatrix} + \mu \begin{bmatrix*}[r]
-2 & 1 & 1\\
1 & -2 & 1\\
1 & 1 & -2
\end{bmatrix*},
\end{equation}
which is a circulant matrix. A circulant matrix is a square matrix whose rows are composed of cyclically shifted versions of the same elements. The eigenvalues (and eigenvectors) of a circulant matrix can be elegantly expressed in terms of these elements and the root of unity (Davis~\cite{circulant}). \com{Particularly,} the eigenvalues of $J$ are given by the expression, 
\begin{multline*}
    \lambda_j = \left(\frac{-1}{1+\alpha+\beta} -2\mu\right) + \left(\frac{-\alpha}{1+\alpha+\beta} +\mu\right)\eta^{j-1} \\
    +\left(\frac{-\beta}{1+\alpha+\beta} +\mu\right)\eta^{2j-2},
\end{multline*}
for $j = 1,2,3$, where
$    \eta = \text{exp}\left(\frac{2\pi i}{3}\right)= -\frac{1}{2} +  \frac{\sqrt{3}}{2}i$.
The eigenvalues can be simplified to 
\begin{subequations}
\begin{align}
    \lambda_1 = &  -1 ,\\ 
    \label{e:equaltriplefp_lambda23}
        \lambda_{2,3} =  & \left[\frac{\alpha+\beta-2}{2(1+\alpha+\beta)}-3\mu \right]\pm \left[\frac{\sqrt{3}(\beta - \alpha)}{2(1+\alpha+\beta)}\right]i. 
\end{align}
\end{subequations}

Thus, a necessary and sufficient condition for stability of the fixed point $e_c$, with the assumption that $1+\alpha+\beta>0$, is
\begin{equation*}
    (\alpha+\beta-2)-6\mu(1+\alpha+\beta) <0,
\end{equation*}
which can be rewritten as 
\begin{equation} \label{e:equaltriplefp_stabilitycond}
    (6\mu-1)(\alpha+\beta)> -6\mu-2.
\end{equation}
 
\begin{figure}[b]
\centering
	\includegraphics[width=0.9\columnwidth]{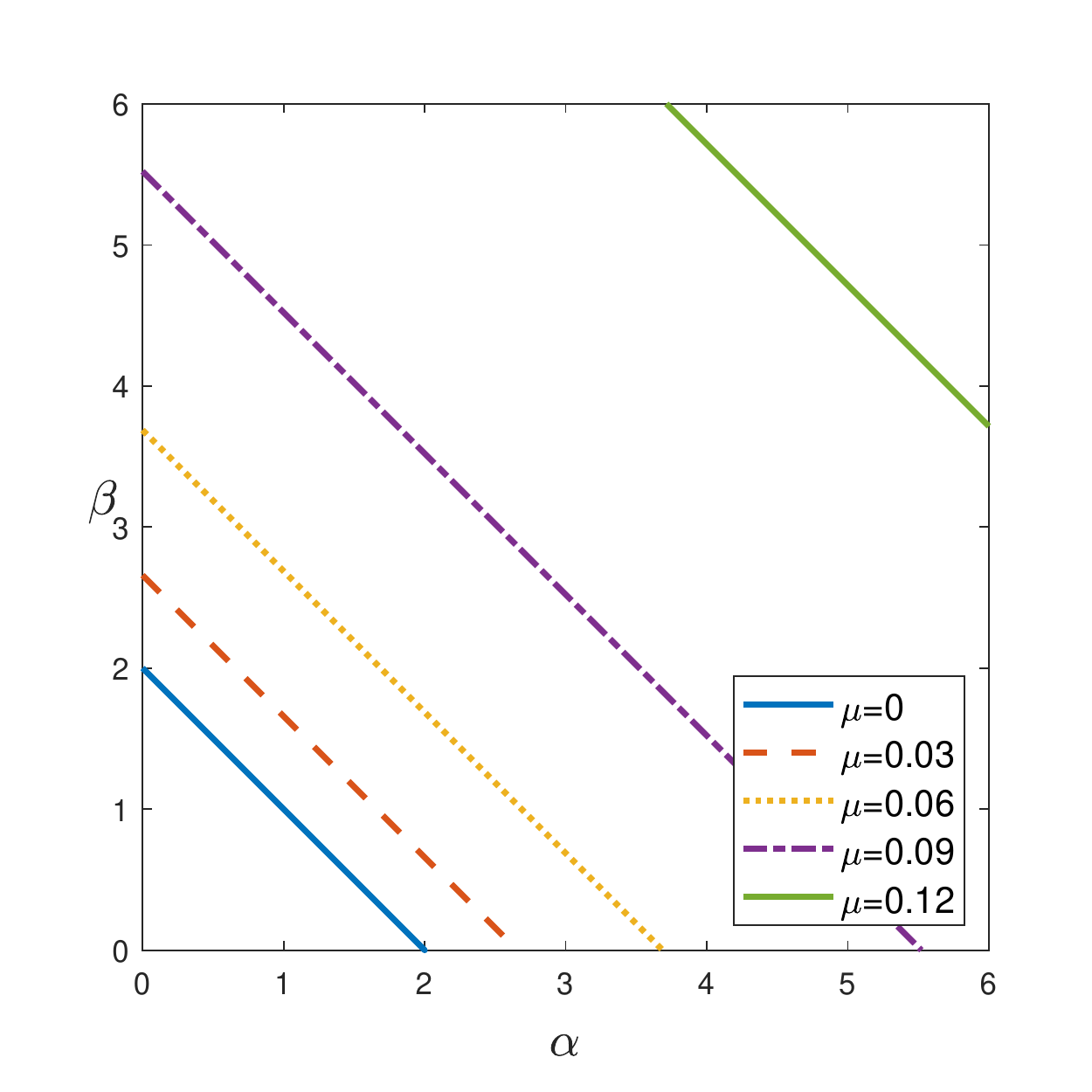}
	\caption{Stability region for the equal-population fixed point $e_c$~\eqref{e:fixedpointsML_equalpop} for various $\mu$ values, \com{with $\mu<1/6$}. Corresponding to the condition \eqref{e:equaltriplefp_stabilitycond_v2}, for a given $\mu$, the equal-population fixed point \com{$e_c$~\eqref{e:fixedpointsML_equalpop}} is stable below the line \eqref{e:HopfBifurcationLine} and is unstable above the line.}
	\label{f:stability_3species}
\end{figure}

If $\mu \geq \frac{1}{6}$, condition \eqref{e:equaltriplefp_stabilitycond} is always satisfied since $\alpha+\beta$ is assumed to be positive, and $e_c$ is always stable for this case. If $\mu<\frac{1}{6}$, the condition becomes
\begin{equation} \label{e:equaltriplefp_stabilitycond_v2}
    \alpha+\beta< \dfrac{6\mu+2}{1-6\mu}.
\end{equation}
Fig.~\ref{f:stability_3species} illustrates the stability regions corresponding to condition \eqref{e:equaltriplefp_stabilitycond_v2} for different values of $\mu$. For values of $\alpha$ and $\beta$ below the line
\begin{equation} \label{e:HopfBifurcationLine}
    \beta_c(\alpha, \mu)= - \alpha + \frac{6 \mu+ 2}{1- 6 \mu} ,
\end{equation}
for a given $\mu$ less than 1/6, the fixed point $e_c$ is stable, and above the line \eqref{e:HopfBifurcationLine}, $e_c$ is unstable. Therefore, the line \eqref{e:HopfBifurcationLine} describes the critical value of the parameter $\beta$ where the equal-population equilibrium $e_c$ undergoes a Hopf bifurcation.
\end{proof}

%%%%%%%%%%%%%%%%%%%%%%%%%%%%%%%%%%%%%%%%%%%%%%%%%%%%%%%%%%%%%%%%%%%
\subsection{Region C}\label{ss:regionC}
Continuing with the analysis of the equal-population fixed point $e_c$~\eqref{e:fixedpointsML_equalpop}, in this subsection we show that this equilibrium undergoes a supercritical Hopf bifurcation for the critical values in \eqref{e:HopfBifurcationLine}.  Consequently, the linearly-perturbed May--Leonard system~\eqref{e:MLperturb} admits a limit cycle solution in Region C. This region is defined to be the two portions of the $\alpha$--$\beta$ plane which are above the line \eqref{e:HopfBifurcationLine} and outside the bifurcation curve labeled 2) in Fig.~\ref{fig:folds}. To prove that the Hopf bifurcation is supercritical, we calculate the first Lyapunov coefficient following the analysis in Kuznetsov~\cite[Chapter 5.4]{kuznetsov1998}
and show that it is negative for all positive values of $\alpha$ and $\beta$.

To set up the notation and simplify the analysis, we first recall how to calculate 
the Lyapunov coefficient for a general $n$-dimensional system of the form
\begin{equation}
    \frac{dx}{dt} = A x + F(x), \quad x \in \R^n.
\end{equation}
In what follows we assume that $A$ is an $n\times n$ matrix that 
has a pair of complex eigenvalues $\lambda = \pm\; i \omega$, where $\omega>0$,
and $F(x) = \rmO(\|x\|^2)$ represents all the nonlinear terms.

The Taylor expansion of $F(x)$ about the origin is given by
\begin{equation} \label{e:Ftaylorexpand}
    F(x) = \frac{1}{2}B(x,x) + \frac{1}{6}C(x,x,x) + \rmO(\|x\|^4) ,
\end{equation}
where
\begin{subequations}
\begin{align}
B_i(x,y) &= \sum_{j,k=1}^n  \left. \frac{\partial^2 F_i(\xi)}{\partial \xi_j \partial \xi_k} \right\rvert_{\xi=0} x_j y_k,\\[4ex]
C_i(x,y,z) &= \sum_{j,k, \ell =1}^n \left. \frac{\partial^3 F_i(\xi)}{\partial \xi_j \partial \xi_k\partial \xi_\ell} \right\rvert_{\xi=0} x_j y_k z_\ell,
\end{align}
\end{subequations}
for $i=1, \cdots,n$. The first Lyapunov coefficient can then be computed as
\begin{multline} \label{eq:firstLyapunov}
    \ell_1(0) = \frac{1}{2 \omega} \Re\Big[ \big\langle p, C(q,q,\bar{q}) \big\rangle 
    -2 \big\langle p, B(q, A^{-1}B(q,\bar{q}) \;) \big\rangle \\
    + \big\langle p, B( \bar{q}, (2 i \omega I_n -A)^{-1} B(q,q) \;) \big\rangle \Big],
\end{multline}
 where the complex vectors $p$ and $q$ satisfy
\begin{equation} \label{e:pqvectorconditions}
    A q = i \omega q, \quad A^Tp = -i \omega p, \quad \langle p, q\rangle = \sum_{i=1}^n \bar{p}_i q_i=1.
\end{equation}

We now proceed to find $\ell_1(0)$ for the linearly-perturbed May--Leonard model~\eqref{e:MLperturb}, where $n=3$, in order to show the following result.

\begin{theorem}
For $\alpha,\:\beta>0$, $\ell_1(0)<0$ in Region C. Therefore, the Hopf bifurcation is supercritical.
\end{theorem}

\begin{proof} Since this system has only quadratic nonlinearities, the formula for the first Lyapunov coefficient stated above~\eqref{eq:firstLyapunov} reduces to
\begin{multline}\label{e:LyapunovCoeff}
    \ell_1(0) = \frac{1}{2 \omega} \Re\Big[ -2 \big\langle p, B(q, A^{-1}B(q,\bar{q}) \;) \big\rangle \\
    + \big\langle p, B( \bar{q}, (2 i \omega I_n -A)^{-1} B(q,q) \;) \big\rangle \Big].
\end{multline}
We first need to change coordinates so that the equal-population equilibrium occurs at the origin. However, because system \eqref{e:MLperturb} has only quadratic nonlinearities,
the coefficients in the Taylor expansion for $F(x)$~\eqref{e:Ftaylorexpand} remain unchanged.
In other words, the vector valued function $B: \R^n \times \R^n \longrightarrow \R^n$ is always the same, 
regardless of whether we compute the Taylor expansion at the origin or at some other point. In particular, it takes the form
\begin{equation}
    B(x,y) = -\begin{bmatrix}
(2 x_1 + \alpha x_2 + \beta x_3) y_1 + \alpha x_1 y_2 + \beta x_1 y_3 \\[1ex]
 \beta x_2 y_1 + (\beta x_1 + 2 x_2 + \alpha x_3) y_2 + \alpha x_2 y_3\\[1ex]
 \alpha x_3 y_1 + \beta x_3 y_2 +(\alpha x_1 +\beta x_2 +2x_3)y_3
\end{bmatrix}.
\end{equation}

The Jacobian of the system evaluated at the equilibrium $e_c$~\eqref{e:fixedpointsML_equalpop} is given by $J$ in \eqref{e:equaltriplefp_jacobian}, which admits the pair of complex eigenvalues (\ref{e:equaltriplefp_lambda23}). Assuming these eigenvalues are purely imaginary, we then get the critical value of $\mu_c$ where the Hopf Bifurcation occurs,
\begin{equation} \label{e:muCHopf}
    \mu_c = \frac{\alpha+\beta -2}{6(1+\alpha+\beta)},
\end{equation}
with the imaginary eigenvalues
\begin{equation} \label{e:Hopf_imaginaryeig}
    \lambda_{2,3} =  \pm \left[\frac{\sqrt{3}(\beta - \alpha)}{2(1+\alpha+\beta)}\right]i \equiv \pm i\omega.
\end{equation}

We first consider the case where $\omega>0$, that is $\beta > \alpha$. Replacing $\mu_c$ by \eqref{e:muCHopf} in $J$ and rewriting its entries in terms of $\omega$~\eqref{e:Hopf_imaginaryeig}, we get the matrix $A$, 
\begin{equation}
    A = -\frac{1}{3} \begin{bmatrix}
 1 & 1 - \sqrt{3} \omega & 1 + \sqrt{3} \omega\\[1ex]
 1 + \sqrt{3} \omega& 1 & 1 - \sqrt{3} \omega \\[1ex]
  1 - \sqrt{3} \omega &1 + \sqrt{3} \omega& 1 
 \end{bmatrix}.
\end{equation}
Since $A$ is a circulant matrix (Davis~\cite{circulant}), the normalized eigenvector corresponding to $i\omega$ is 
\begin{equation}
    q= \frac{1}{\sqrt{3}}
 \begin{bmatrix}
1 \\
\eta\\
\eta^2
\end{bmatrix}, \quad \text{ where }\eta = -\frac{1}{2} +  \frac{\sqrt{3}}{2}i.
\end{equation}

This is also the normalized eigenvector of $A^T$, also a circulant matrix, corresponding to $-i\omega$. Hence, we take 
\begin{equation}
    p = q =  \frac{1}{\sqrt{3}}
 \begin{bmatrix}
1 \\
-\frac{1}{2}+\frac{\sqrt{3}}{2}i\\
-\frac{1}{2}-\frac{\sqrt{3}}{2}i
\end{bmatrix},
\end{equation}
which satisfies \eqref{e:pqvectorconditions}.
Continuing the computations using Mathematica, the real part of the first term of the first Lyapunov coefficient, $-2 \langle p, B(q,A^{-1}B(q,\bar{q})\rangle$, results in $-1/3\, (\alpha + \beta-2) (4 + \alpha + \beta)$. The second term, $$\left\langle p, B\left( \bar{q}, (2 i \omega I_n -A)^{-1} B(q,q) \right) \right\rangle,$$ turns out to be a purely imaginary number. Replacing these expressions in \eqref{e:LyapunovCoeff} leads to 
\begin{multline}
    \ell_1(0) =  - \frac{(\alpha+\beta-2)(\alpha+\beta+4)}{6 \omega}, \\
    \text{ where } \omega = \frac{\sqrt{3}(\beta-\alpha)}{2(\alpha+\beta+1)}.
\end{multline}
In Region C, we know that $\alpha+\beta >2,$ so $\ell_1(0)<0$. Therefore, the Hopf bifurcation is supercritical.

In the case where $\omega<0$, that is $\beta < \alpha$, we get $\lambda_{2,3}= \mp i \lvert\omega\rvert$ and the eigenvectors can be chosen to be the conjugates of those in the previous case. We then get a similar expression for $\ell_1(0)$,
\begin{equation}
    \ell_1(0) =  - \frac{(\alpha+\beta-2)(\alpha+\beta+4)}{6 \lvert\omega\rvert} ,
\end{equation}
which again is negative in Region C.
\end{proof}

Notice that when $\alpha = \beta$ the value of $\omega =0$. In this case, the Jacobian evaluated at the equal-population equilibrium, $e_c$~\eqref{e:fixedpointsML_equalpop}, is given by

\begin{equation}
    A = - \frac{1}{3} 
\begin{bmatrix}
1 & 1 & 1\\
1 & 1 & 1\\
1 & 1 & 1
\end{bmatrix} ,
\end{equation}
and has two simple zero eigenvalues.

\begin{figure}[t] %  figure placement: here, top, bottom, or page
    \centering
    \subfigure[$\alpha = 1.3$, $\beta = 1.3585$ \label{fig:RegCP_CuspSim_Inside}]{\includegraphics[width=0.7\columnwidth]{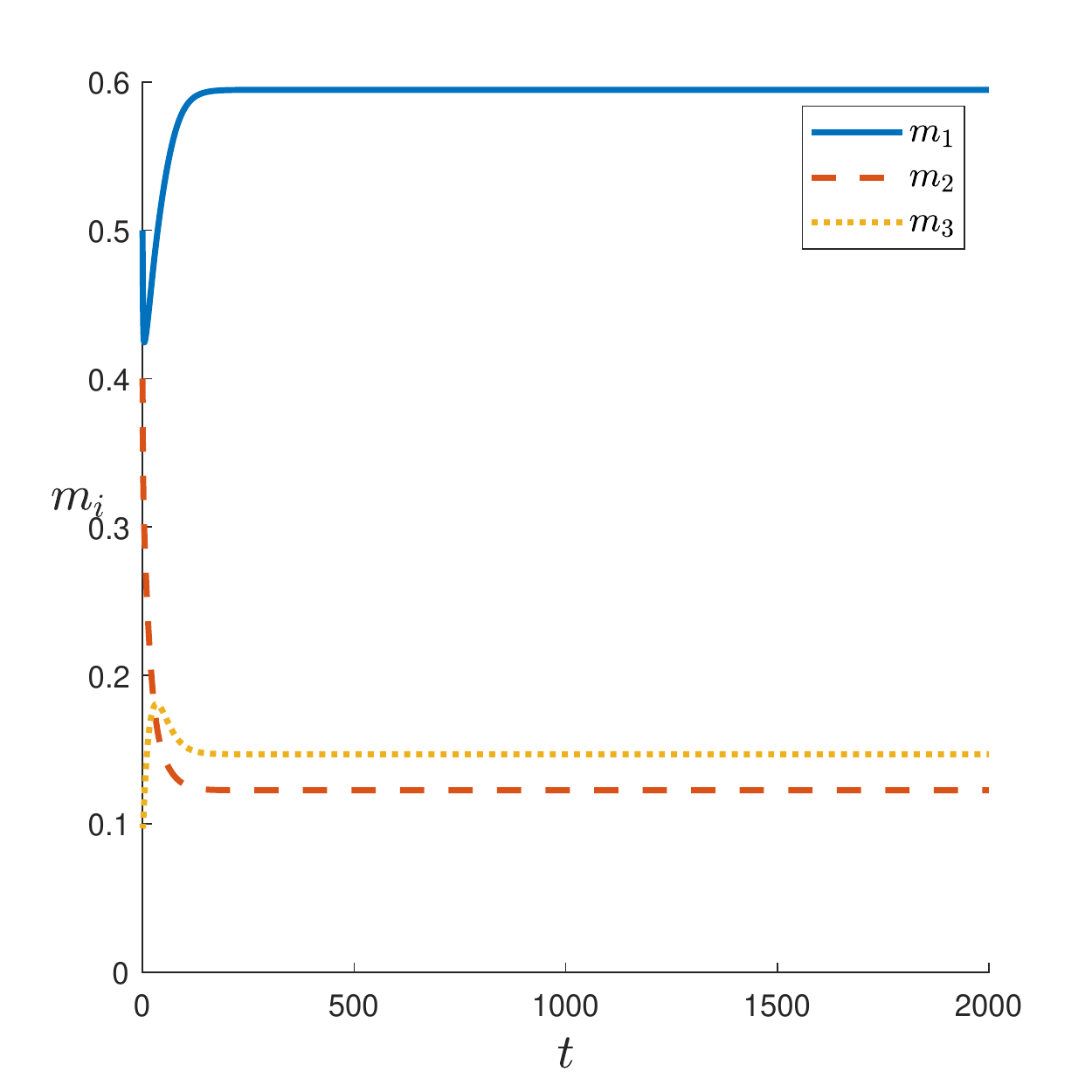}}
    \hspace{0.2in}
    \subfigure[$\alpha = 1.2$, $\beta = 1.4585$ \label{fig:RegCP_CuspSim_Outside}]{\includegraphics[width=0.7\columnwidth]{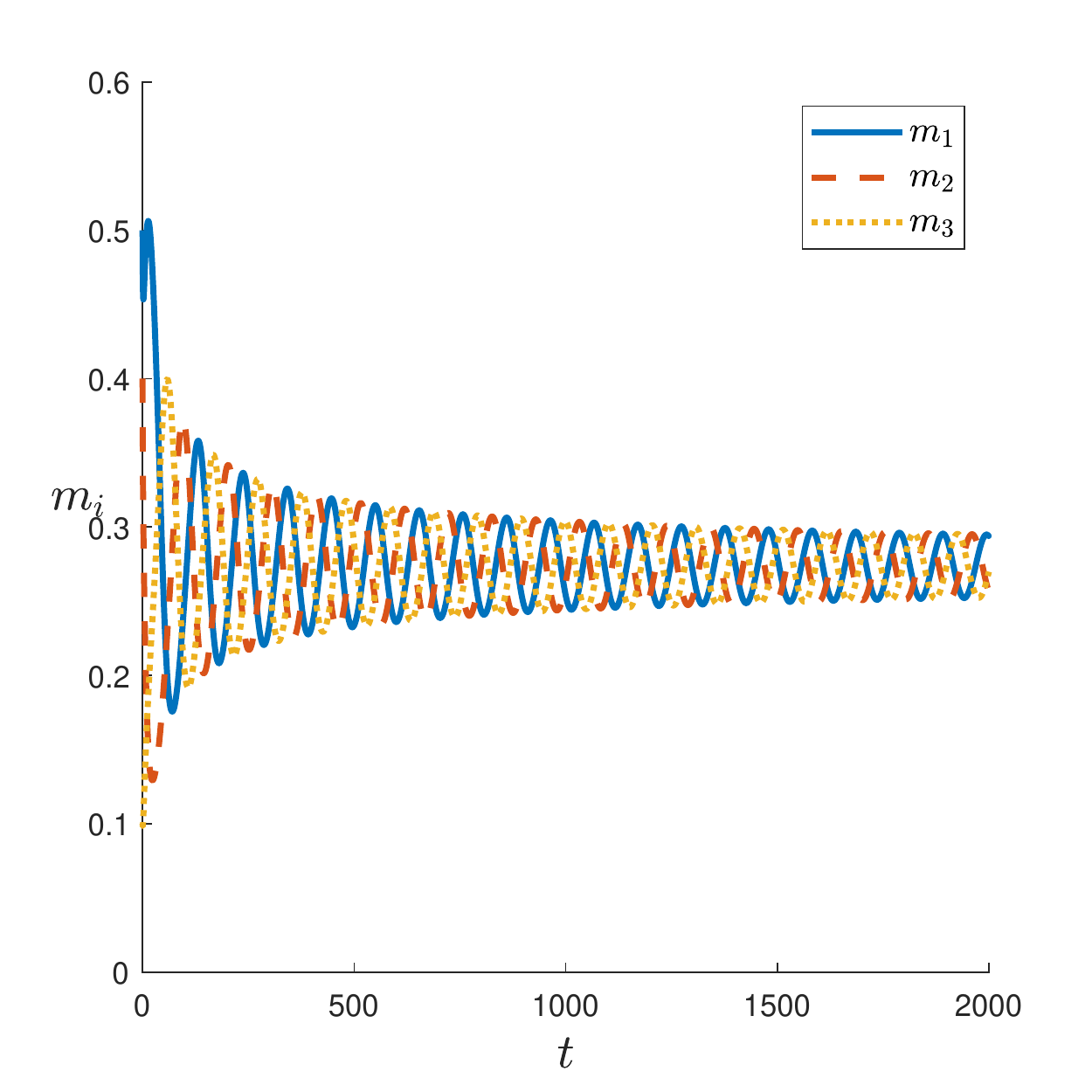}}
 \caption{Solutions corresponding to different $\alpha$ and $\beta$ values in Region C$^{\prime}$ (which corresponds to the line \eqref{e:HopfBifurcationLine}) with $\mu=0.03$ and initial condition ${\bf m} =(0.5,  0.4, 0.1)$. We denote the part of the line \eqref{e:HopfBifurcationLine} where $\alpha \in [1.272, 1.3865]$ and $\beta = -\alpha +2.6585$ as ``inside'' the fold bifurcation curve; this is the boundary between Regions B and D (see Fig.~\ref{f:phasediagram-MLRPS}). (a)~Inside the curve, the solution tends to a triple-population fixed point. (b)~Outside the curve, the solution is periodic.}
\label{fig:RegCP_CuspSim}
\end{figure}

Thus, with the exception of the  point where $\alpha = \beta$, the above calculations are valid for almost all parameter values along the line $\beta_c(\alpha, \mu)$~\eqref{e:HopfBifurcationLine} and show that the Hopf bifurcation is supercritical. The analysis, however, does not distinguish between sections of the Hopf line that lie adjacent to Region  B (where the dynamics tends to a stable fixed point, as illustrated later in Fig.~\ref{fig:RegCP_CuspSim_Inside}) and those sections that are next to Region C (where we observe limit cycles, as illustrated in Fig.~\ref{fig:RegCP_CuspSim_Outside}).

Therefore, in order to justify the existence of limit cycles close to the Hopf bifurcation line~\eqref{e:HopfBifurcationLine} and within Region C, we first notice that in this part of parameter space the only equilibria that are present in the system are the equal-population fixed point, $e_c$~\eqref{e:fixedpointsML_equalpop}, and the fixed point at the origin, $e_0$~\eqref{e:fixedpointsML_zero}, both of which are unstable. In addition, we know that at the bifurcation point, the equal-population fixed point has two center directions and one stable direction. As a result, the center manifold for this equilibrium is attracting, and  the dynamics of the system near this point will remain in this locally invariant manifold. Because the bifurcation is supercritical, we then know that a limit cycle is formed.

On the other hand, we know that in sections of parameter space that are at the intersection of the Hopf bifurcation line~\eqref{e:HopfBifurcationLine} and Region B, the system has an additional six fixed points that emerge from the fold line (Fig.~\ref{fig:folds}), three of which are stable (Section~\ref{ss:regionB}). We suspect the  presence of these stable fixed points is what prevents the system from forming a  limit cycle, but we do not have a general proof for this result. 
We can, however, confirm using AUTO that in this region no periodic orbits bifurcate from the Hopf point.
This also holds in the degenerate case, when the parameters $\alpha$ and $\beta$ lie on the part of the Hopf line that borders Region B, and which in addition satisfy  $\alpha = \beta$. We explore the dynamics of the system for these particular values of the parameters in the following subsection.
 
%\newpage
%%%%%%%%%%%%%%%%%%%%%%%%%%%%%%%%%%%%%%%%%%%%%%%%%%%%%%%%%%%%%%%%%%%
\subsection{Region C$^\prime$}\label{ss:regionCprime}

Region C$^\prime$ encompasses the points lying on the Hopf bifurcation line~\eqref{e:HopfBifurcationLine}. To determine the system's dynamics at the degenerate point  ($\alpha = \beta$), which also lies on the line,
we perform a change of coordinates that highlights the periodic structure inherent in the system. In particular, we use the generalized cylindrical coordinates, which were introduced to study the time evolution of nonperiodic oscillations of the May--Leonard model \eqref{e:ML} in Phillipson~\cite{phillipson1984} and Phillipson et al.~\cite{phillipson-etal}.

We translate the linearly-perturbed May--Leonard equations~\eqref{e:MLperturb} to coordinates $x_i$ with respect to the equal-population fixed point $e_c$~\eqref{e:fixedpointsML_equalpop}, i.e., $x_i(t) = m_i(t)-1/(1+\alpha+\beta)$, and then utilize the generalized cylindrical coordinates $R$, $\theta$, and $Z$, via the transformation
\begin{subequations} \label{e:cylindricaltransformation}
\begin{align}
    x_1 &= 2 R \cos \theta +Z , \\
    x_2 &= -R \cos \theta -\sqrt{3} R \sin \theta + Z , \\
    x_3 &= -R \cos \theta + \sqrt{3} R \sin \theta + Z .
\end{align}
\end{subequations}
With these transformations, the linearly-perturbed May--Leonard model~\eqref{e:MLperturb} becomes
\begin{subequations} \label{e:MLperturb_cylindrical}
\begin{align}
    \frac{dR}{dt} &= (\lambda-3\mu) R -\sigma R^2\Big( \omega \sin(3\theta)-\lambda \cos(3\theta) \Big) \nonumber \\
    & \qquad - \left(\frac{\sigma+3}{2}\right) R Z , \label{e:MLperturb_cylindrical_R} \\
    \frac{d\theta}{dt} &= \omega - \sigma R\Big(\omega \cos(3\theta)+\lambda\sin(3\theta) \Big) + \sigma \omega Z , \label{e:MLperturb_cylindrical_theta} \\
    \frac{dZ}{dt} &= -Z -\sigma Z^2 + 2 \lambda \sigma R^2 , \label{e:MLperturb_cylindrical_Z}
\end{align}
\end{subequations}
where
\begin{equation}
    \sigma=1+\alpha+\beta, \quad \lambda =\frac{\alpha+\beta-2}{2 \com{(1+\alpha+\beta)}}, \quad \omega=\frac{\sqrt{3}(\beta-\alpha)}{2\com{(1+\alpha+\beta)}} .
\end{equation}
The only difference between \eqref{e:MLperturb_cylindrical} and the May--Leonard model~\eqref{e:ML} in cylindrical coordinates (Phillipson et al.~\cite{phillipson-etal}) is the first term in the $dR/dt$ equation~\eqref{e:MLperturb_cylindrical_R}, \mbox{$(\lambda-3\mu)R$}. Thus, since the other two equations \eqref{e:MLperturb_cylindrical_theta}--\eqref{e:MLperturb_cylindrical_Z} are exactly the same, we refer the reader to Phillipson et al.~\cite{phillipson-etal} for the case when $\mu=0$, and here we analyze how finding the fixed points of the cylindrical coordinates system \eqref{e:MLperturb_cylindrical} is modified when $\mu\neq0$.

Since Region C$^{\prime}$ is defined to be the Hopf bifurcation line $\beta_c(\alpha,\mu)$~\eqref{e:HopfBifurcationLine}, we set
\begin{equation} \label{e:HopfBifurcationLine_mu}
    \mu=\mu_c(\alpha,\beta) = \frac{\alpha+\beta-2}{6(1+\alpha+\beta)} ,
\end{equation}
which is equivalent to setting $\mu=\lambda/3$. Thus, the first term in the $dR/dt$ equation~\eqref{e:MLperturb_cylindrical_R} vanishes, and the system under investigation is 
\begin{subequations} \label{e:MLperturb_cylindrical_withmufixed}
\begin{align}
    \frac{dR}{dt} &= -\sigma R^2\Big( \omega \sin(3\theta)-\lambda \cos(3\theta) \Big) - \left(\frac{\sigma+3}{2}\right) R Z , \\
    \frac{d\theta}{dt} &= \omega - \sigma R\Big(\omega \cos(3\theta)+\lambda\sin(3\theta) \Big) + \sigma \omega Z , \label{e:MLperturb_cylindrical_withmufixed_theta} \\
    \frac{dZ}{dt} &= -Z -\sigma Z^2 + 2 \lambda \sigma R^2 .
\end{align}
\end{subequations}

We propose the following theorem regarding the dynamics of system~\eqref{e:MLperturb_cylindrical_withmufixed} in Region C$^{\prime}$.

\begin{theorem}
Fixed point solutions of \eqref{e:MLperturb_cylindrical_withmufixed} exist if and only if $\omega=0$ (equivalently, $\alpha=\beta$).
\end{theorem}

\begin{proof}

Trivial fixed points of the system \eqref{e:MLperturb_cylindrical_withmufixed} are of the form $(R^*,\theta^*,Z^*)=(0,\theta,0)$, where the angular coordinate $\theta$ is arbitrary and the condition $\omega=0$ must be satisfied. Since $R$ is a radial coordinate and $Z$ is a cylindrical coordinate, this corresponds to a stationary fixed point in the original coordinates $(m_1,m_2,m_3)$. In fact, by \eqref{e:cylindricaltransformation} and the definition of $x_i$, it corresponds to the equal-population fixed point, $e_c$~\eqref{e:fixedpointsML_equalpop}.

Positive fixed points of the system \eqref{e:MLperturb_cylindrical_withmufixed} with $\omega=0$ (equivalently, $\alpha=\beta$) are of the form
\begin{equation} \label{e:cylindrical_fixedpts_mufixed}
    (R^*,\theta^*,Z^*) = \left( \frac{\beta+2}{9(\beta+1)} , \quad \frac{\pi n}{3}, \quad \frac{\beta-1}{9(\beta+1)} \right), \quad n\in\mathbb{Z} .
\end{equation}
 Fig.~\ref{fig:RthetaZsimulation_muperturb_zero} indicates that solutions in the generalized cylindrical coordinates do not oscillate, but rather tend to one of the six fixed points given in~\eqref{e:cylindrical_fixedpts_mufixed}.
 These steady states correspond to the equilibria in the original $(m_1,m_2,m_3)$-coordinates that emerge from the fold bifurcation.

\begin{figure}[h] %  figure placement: here, top, bottom, or page
    \centering
    \subfigure[$\mu=\mu_c$ \label{fig:RthetaZsimulation_muperturb_zero}]{\includegraphics[width=0.7\columnwidth]{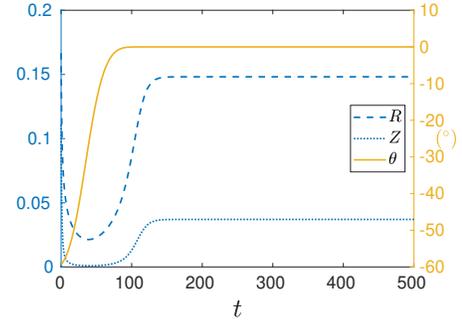}}
    \subfigure[$\mu=\mu_c+0.003$ \label{fig:RthetaZsimulation_muperturb_medium}]{\includegraphics[width=0.7\columnwidth]{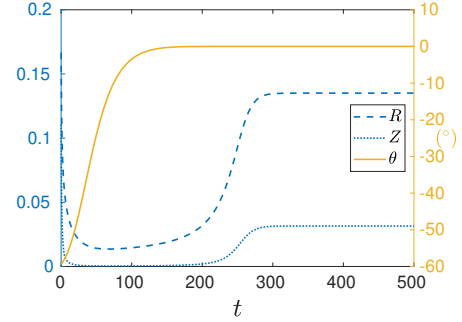}}
    \subfigure[$\mu=\mu_c+0.005$ \label{fig:RthetaZsimulation_muperturb_large}]{\includegraphics[width=0.7\columnwidth]{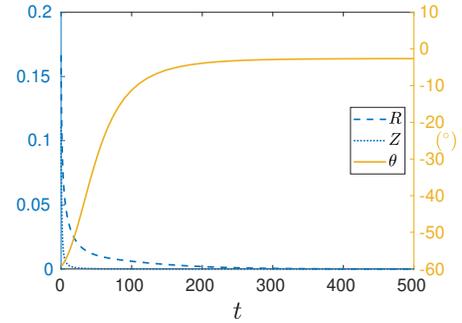}}
\caption{Simulation of the cylindrical coordinates system \eqref{e:MLperturb_cylindrical} with $\alpha=\beta=2$, $\mu_c=\mu_c(\beta,\beta)=0.0667$, and initial conditions $(R_0,\theta_0,Z_0)=(0.1666, -1.0372, 0.11075)$. The left vertical axis corresponds to $R$ and $Z$ and the right vertical axis corresponds to $\theta$, in degrees. (a)~$\mu=\mu_c=0.0667$: Positive fixed points are of the form \eqref{e:cylindrical_fixedpts_mufixed} with $(R^*,\theta^*,Z^*)=(0.14815,0,0.037037)$. The numerical simulation at $t=500$ predicts the same values. (b)~$\mu=\mu_c+0.003=0.0697$: Fixed points are of the form \eqref{e:MLperturb_cylindrical_fixedpts} with $(R^*_{-},Z^*_{-}) = (0.0097616,0.0001904)$ and $(R^*_{+},Z^*_{+}) = (0.13505,0.031513)$. The numerical simulation at $t=500$ predicts $(R^*,Z^*)=(R^*_{+},Z^*_{+})$. (c)~$\mu=\mu_c+0.005=0.0717$: The positive fixed point is lost and the system tends to $(R^*,\theta^*,Z^*)=(0,\theta,0)$ for an arbitrary $\theta$.}
\label{fig:RthetaZsimulation}
\end{figure}

Now assume that $\omega$ is not necessarily zero. First, notice that
the fixed points expression \eqref{e:cylindrical_fixedpts_mufixed} is similar in form to the fixed points expression found for the May--Leonard model~\eqref{e:ML} in generalized cylindrical coordinates (Phillipson et al.~\cite{phillipson-etal}). However, in parameter regions for which fixed points of the May--Leonard model in generalized cylindrical coordinates exist, the angular coordinate $\theta^*$ is equally realizable and the system never settles down to a fixed value for $\theta$. Here, since there are no oscillations in the solution (Fig.~\ref{fig:RthetaZsimulation_muperturb_zero}), the system tends to just one angular coordinate $\theta^*$.

To determine the relationship between $\beta$ and $\mu$ for which positive fixed points of the system \eqref{e:MLperturb_cylindrical} exist, Fig.~\ref{fig:RthetaZsimulation} illustrates that as $\mu$ increases, there is a range for which a positive fixed point exists (Fig.~\ref{fig:RthetaZsimulation_muperturb_medium}) and then vanishes (Fig.~\ref{fig:RthetaZsimulation_muperturb_large}). To find an analytic expression for these bounds, we first observe that since $\theta$ appears in the cylindrical coordinates system \eqref{e:MLperturb_cylindrical} in the trigonometric arguments as $(3\theta)$, fixed point solutions will require $\theta=\pi n/3$ for $n\in\mathbb{Z}$. Solving for the fixed points under the assumption that $\theta=0$, without loss of generality, we find that we must have $\omega=0$ for the equations to be satisfied. (Since the right-hand side of \eqref{e:MLperturb_cylindrical_withmufixed_theta} is a product of $\omega$ with another factor, if we take that factor to be equal to 0 and set $\omega\neq 0$, then we find that the fixed point satisfies $R^*=1/3$ and $Z^*=0$ with the condition that $\alpha+\beta=2$, which is equivalent to $\mu=0$. Thus, that system is equivalent to the May--Leonard model~\eqref{e:ML}.)
\end{proof}

Setting $\theta=0$ and $\alpha=\beta$ in \eqref{e:MLperturb_cylindrical}, we find that the fixed points satisfy
\begin{subequations} \label{e:MLperturb_cylindrical_fixedpts}
\begin{align}
    R^*_{\pm} &= \frac{\beta}{6 (1 + \beta)} - \frac{(1 + 2 \beta)\mu}{3 (1 + \beta)} \nonumber \\
    & \qquad \pm \frac{\sqrt{(\beta - 1) (2 + \beta)^2 (8 \mu^2 - 1 + \beta (1 - 4 \mu)^2)}}{6 (\beta^2 - 1)}, \\
    Z^*_{\pm} &= \frac{\beta - 1}{(2 + \beta) (1 + 2 \beta)} - \frac{3 \mu}{2 + \beta} + \frac{\beta - 1}{2 + \beta} R^*_{\pm} .
\end{align}
\end{subequations}
Hence, for a given $\beta \neq 1$ and $\mu$ and with $\theta=0$, there are at most two positive fixed points given by $(R^*,\theta^*,Z^*) = (R^*_{+},0,Z^*_{+})$ and $(R^*,\theta^*,Z^*) = (R^*_{-},0,Z^*_{-})$.
We plot the region in $\beta$--$\mu$ space for which both of these fixed points are nonnegative in Fig.~\ref{fig:cylindricalfixedpts}.

Furthermore, we observe in Fig.~\ref{fig:RthetaZsimulation_muperturb_zero}--\ref{fig:RthetaZsimulation_muperturb_medium} that $(R^*_{-},Z^*_{-})$ is attained as a local minimum early in the simulations and $(R^*_{+},Z^*_{+})$ is the maximum value attained as $t$ increases. For values of $\beta$ and $\mu$ that are within the shaded region illustrated in Fig.~\ref{fig:cylindricalfixedpts}, simulations in Fig.~\ref{fig:RthetaZsimulation_muperturb_zero}--\ref{fig:RthetaZsimulation_muperturb_medium} indicate that the system starts in the direction toward one location, but then tends toward a second location in the long term.
This reflects the fact that the fixed points $(R^*_{-},Z^*_{-})$  are unstable, while the fixed points $(R^*_{+},Z^*_{+})$ are stable.
Thus, we do not observe any oscillations in the generalized cylindrical coordinates system~\eqref{e:MLperturb_cylindrical}, let alone nonperiodic oscillations as observed in the May--Leonard model~\eqref{e:ML}.

\begin{figure}[h] %  figure placement: here, top, bottom, or page
    \centering
    \includegraphics[width=0.8\columnwidth]{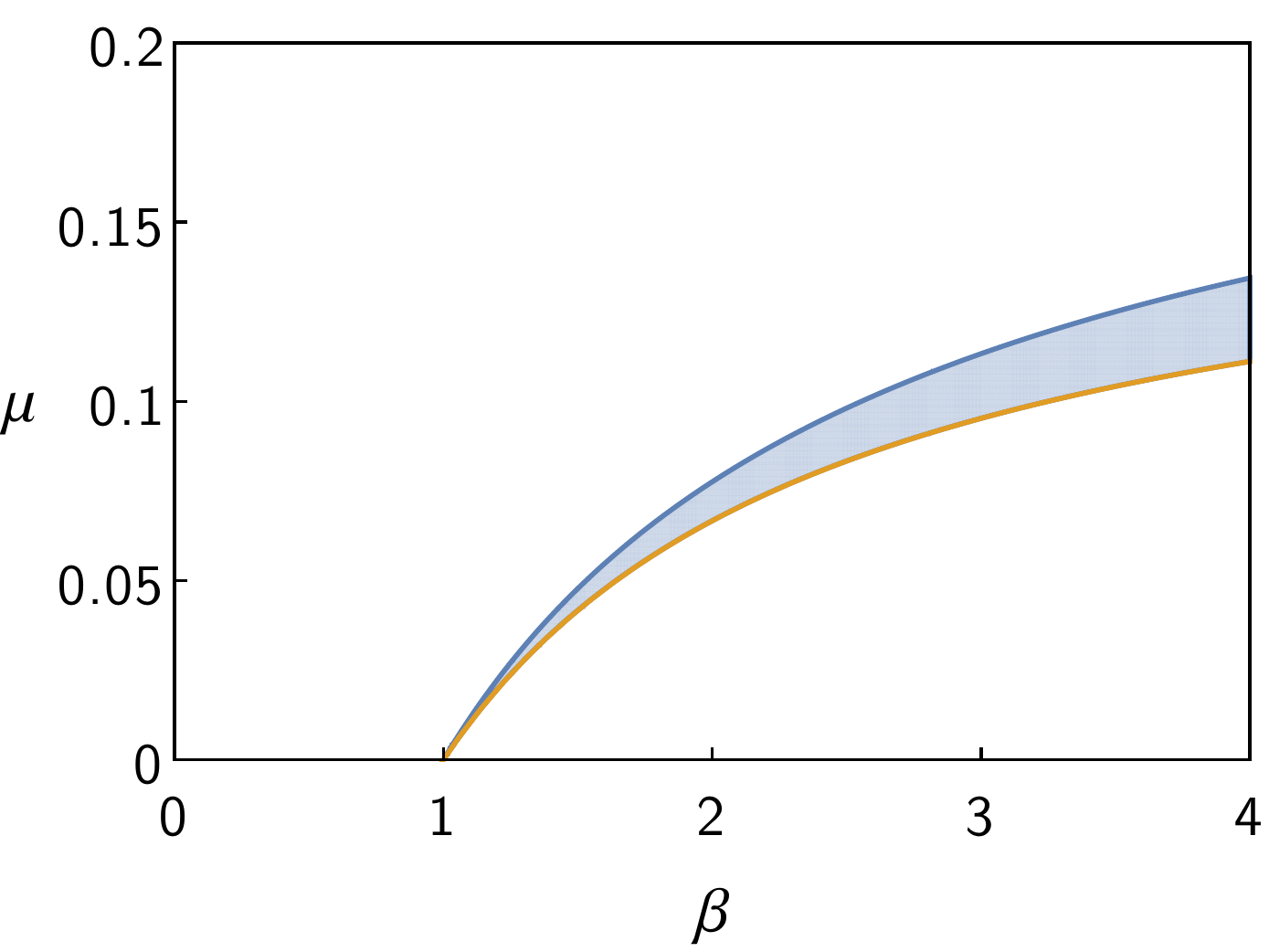}
\caption{Nonnegative fixed points in $\beta$--$\mu$ space. The shaded region shows the values of $\mu$ and $\beta$ for which all $R^*_{\pm}$ and $Z^*_{\pm}$ \eqref{e:MLperturb_cylindrical_fixedpts} are nonnegative, and thus system \eqref{e:MLperturb_cylindrical} with $\theta=0$ has nonnegative fixed points. The \com{orange} lower bounding curve is \eqref{e:HopfBifurcationLine_mu} with $\alpha=\beta$, i.e., $\mu=\mu_c(\beta,\beta)=(\beta-1)/(3(1+2\beta))$.}
\label{fig:cylindricalfixedpts}
\end{figure}

%%%%%%%%%%%%%%%%%%%%%%%%%%%%%%%%%%%%%%%%%%%%%%%%%%%%%%%%%%%%%%%%%%%
\subsection{Region B}\label{ss:regionB}

In this subsection, we focus on the number and stability of fixed points that exist for parameter values in Region B.

These fixed points then coalesce in a second fold bifurcation when $\alpha$ is increased even further. Due to the symmetries inherent in the linearly-perturbed May--Leonard model~\eqref{e:MLperturb}, we conclude that there are actually three fold bifurcations that occur along these curves. 
As a result, in addition to the fixed point at the origin $e_0$~\eqref{e:fixedpointsML_zero} and the equal-population equilibrium $e_c$~\eqref{e:fixedpointsML_equalpop} (which are present for all values $\alpha, \beta>0$ and are both unstable in this region),
there are six other positive steady states inside Region B.

These six new equilibria also depend on $\mu$ in an interesting way. While one set of fixed points can be traced back to a family of single-population equilibria, \eqref{e:fixedpointsML_singlepop}, as $\mu \to 0$, the other set originates from a family of dual-population fixed points, \eqref{e:fixedpointsML_dualpop}. Indeed, this is confirmed in 
Appendix~\ref{sec:appendix}, where we compute a second-order approximation in $\mu$ for these steady states using a perturbation analysis.
In Fig.~\ref{fig:regionb}, we compare our analytic results with those obtained numerically using AUTO 07~\cite{auto07p} for parameter values $\mu =0.03$, $\beta =2$, and $\alpha \in (1,7)$. 
The continuation curve, which plots the $m_1$-component of all six equilibria versus the parameter $\alpha$, is represented as a black solid curve, while the second-order approximations are shown as dashed curves.
The blue dash-dotted curves correspond to components of the steady state that can be traced back to a single-population fixed point, while the red dashed curves emerge from a dual-population equilibrium.
Fig.~\ref{fig:regionb} illustrates that the second-order approximations (Appendix~\ref{sec:appendix}) closely estimate the numerically calculated continuation curves for values of $\alpha$ in a neighborhood of 3.% in a neighborhood of 3.

 \begin{figure}[h] %  figure placement: here, top, bottom, or page
   \centering
  \subfigure[]{ \includegraphics[width=0.8\columnwidth]{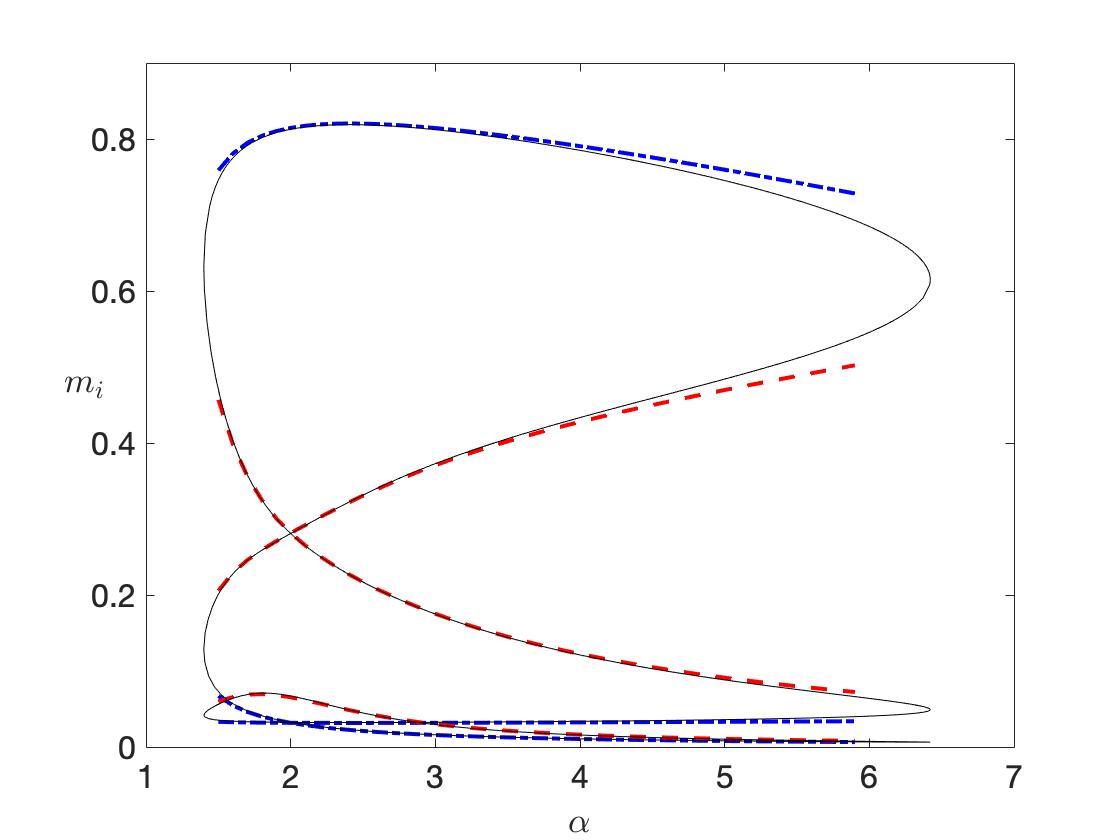} }
   \subfigure[]{\includegraphics[width=0.8\columnwidth]{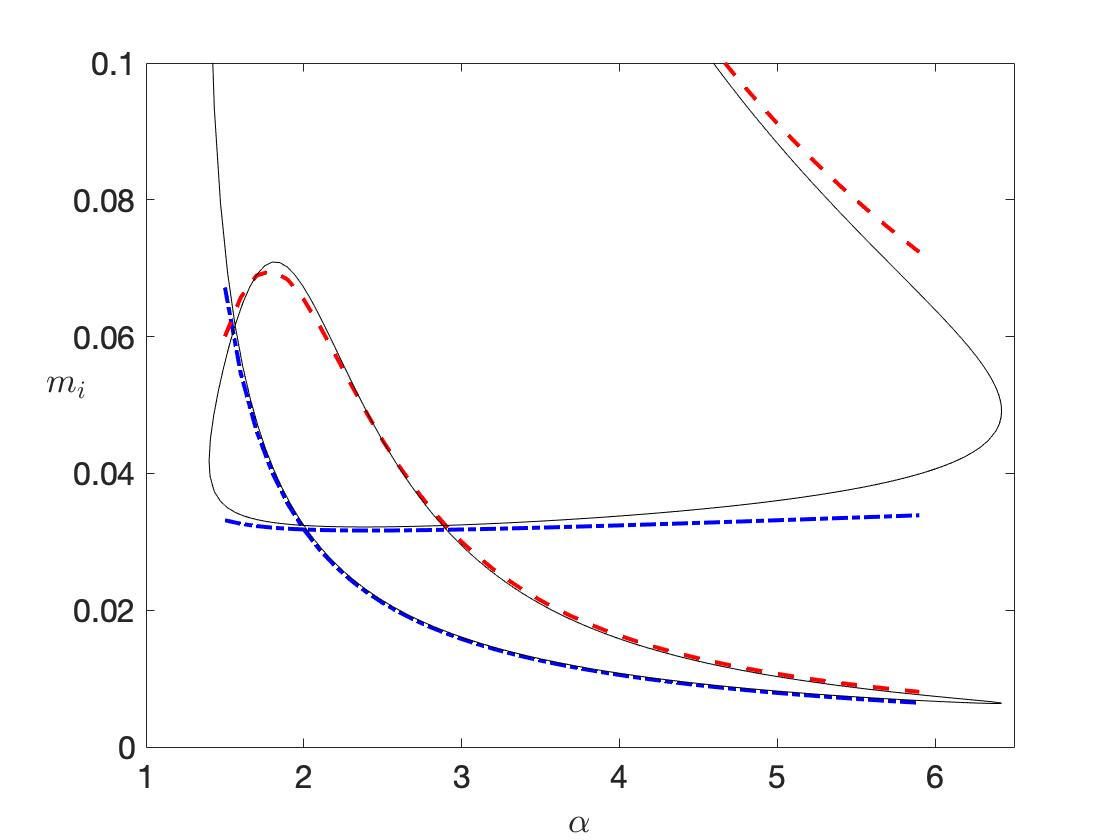} }
\caption{ (a) Continuation curves for the equilibria that emerge from the fold bifurcation obtained by varying $\alpha$, with $\mu =0.03$ and $\beta =2$ fixed. 
The solid curve represents the continuation curve obtained using AUTO 07~\cite{auto07p}. Dashed curves represent the second-order approximation
in $\mu$ (Appendix~\ref{sec:appendix}) for the fixed points that originated from a single-population equilibrium (blue dash-dotted curves)
and from a dual-population equilibrium (red dashed curves). (b) Zoom in on the bottom region of (a).}
\label{fig:regionb}
\end{figure}

Due to the symmetries present in the system, the continuation curve 
is also a plot of the other two components, $m_2$ and $m_3$, of the two fixed points that emerge from the fold bifurcation when $\beta=2$ and $\alpha$ is small.

The first bifurcation at $\alpha\approx 1.3993$  corresponds to the left-most leg of curve 2) shown in Fig.~\ref{fig:folds}. Then, as the value of $\alpha$ is increased, these steady states coalesce in a second fold bifurcation at $\alpha\approx 6.4363$. This corresponds to the right-most leg of curve 2) in Fig.~\ref{fig:folds}.

We also studied the stability of these six new fixed points numerically. 
These results are summarized in Fig.~\ref{fig:stabregionb}, where we see that the equilibria 
that emerged from the single-population steady state~\eqref{e:fixedpointsML_singlepop} are stable, while the equilibria corresponding to
the dual-population fixed point~\eqref{e:fixedpointsML_dualpop} have one unstable direction. As a result, the dynamics inside Region B are determined by the initial conditions.

\begin{figure}[h] %  figure placement: here, top, bottom, or page
   \centering
   \subfigure[]{\includegraphics[width=0.8\columnwidth]{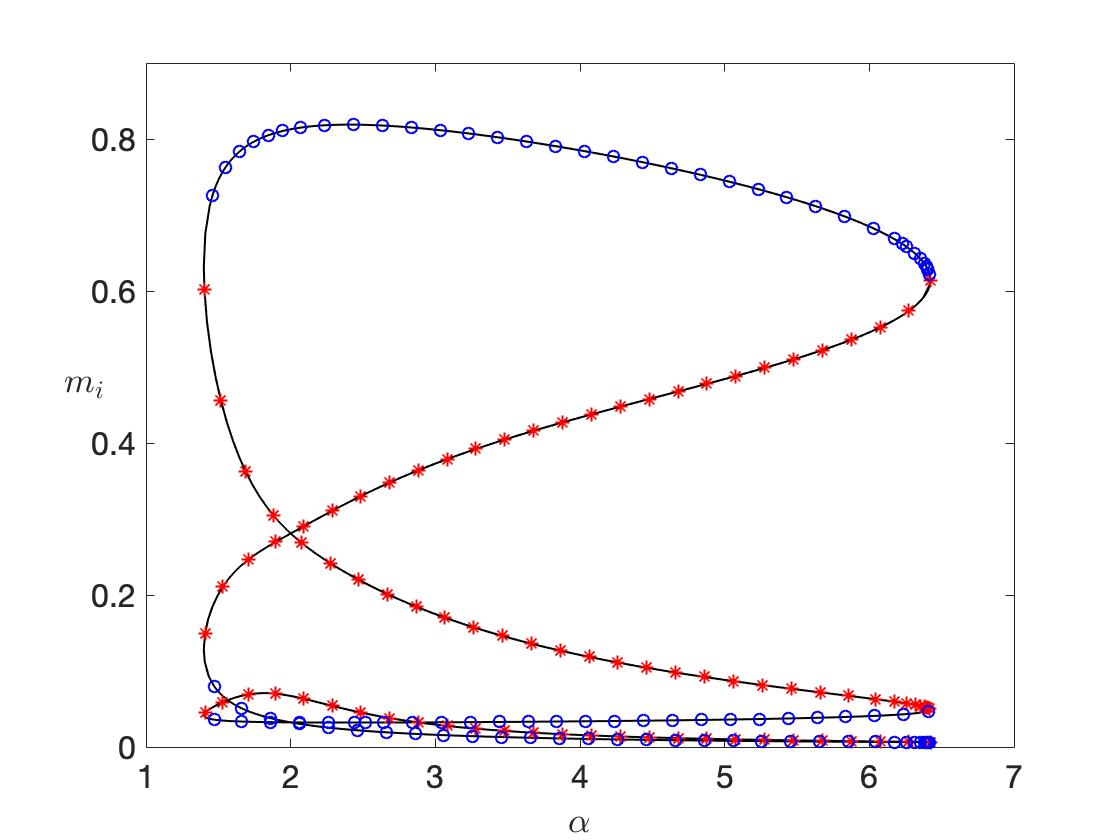}} 
   \subfigure[]{ \includegraphics[width=0.8\columnwidth]{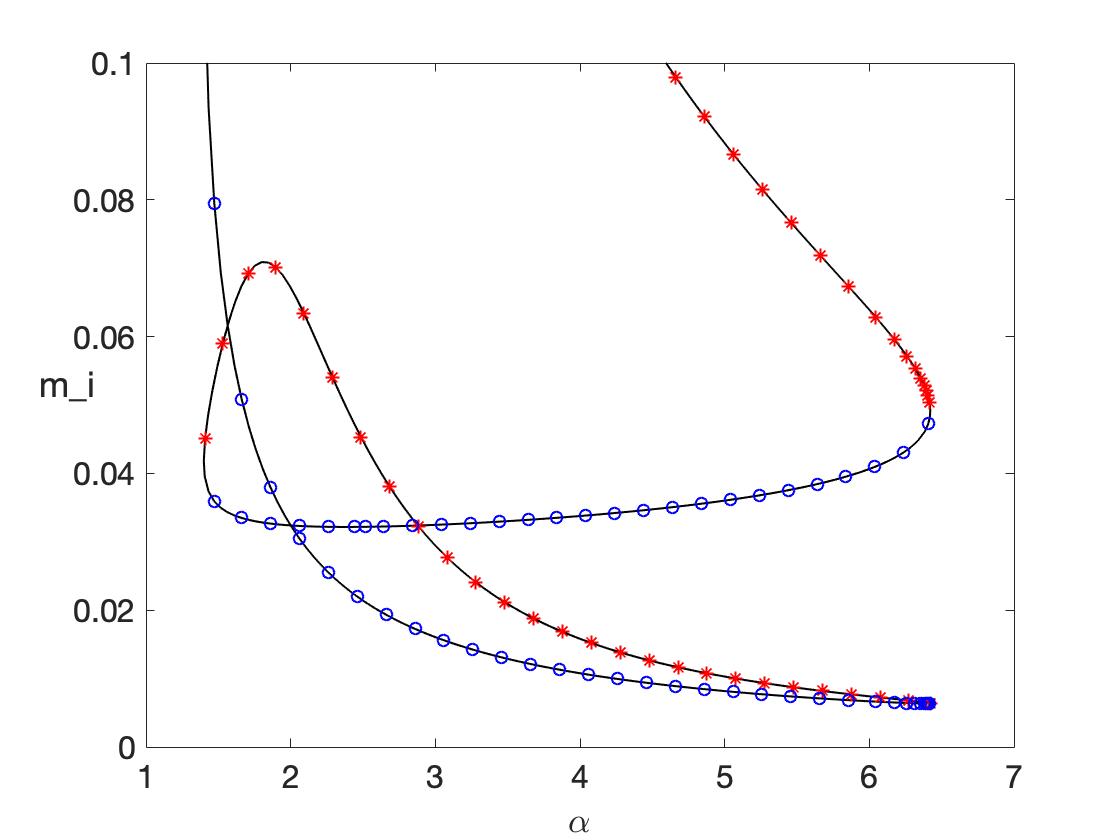} }
   \caption{(a) Stability of equilibria that emerge from \com{the} fold bifurcation obtained by varying $\alpha$, with 
$\mu =0.03$ and $\beta =2$ fixed. 
The solid curve represents the continuation curve obtained using AUTO~07~\cite{auto07p}. Open blue circles represent stable equilibria, 
while red stars represent equilibria with one unstable direction. (b) Zoom in on the bottom region of (a).}
\label{fig:stabregionb}
\end{figure}

\com{
\begin{Remark}\label{r:heteroclinic}
   We notice that for positive values of the parameter $\mu$, the linearly-perturbed May--Leonard model~\eqref{e:MLperturb} does not possess heteroclinic cycles. We subsequently explain this behavior.
\end{Remark}
}

\begin{figure}[h] %  figure placement: here, top, bottom, or page
   \centering
   \subfigure[\label{fig:regionb2_subplota}]{\includegraphics[width=0.8\columnwidth]{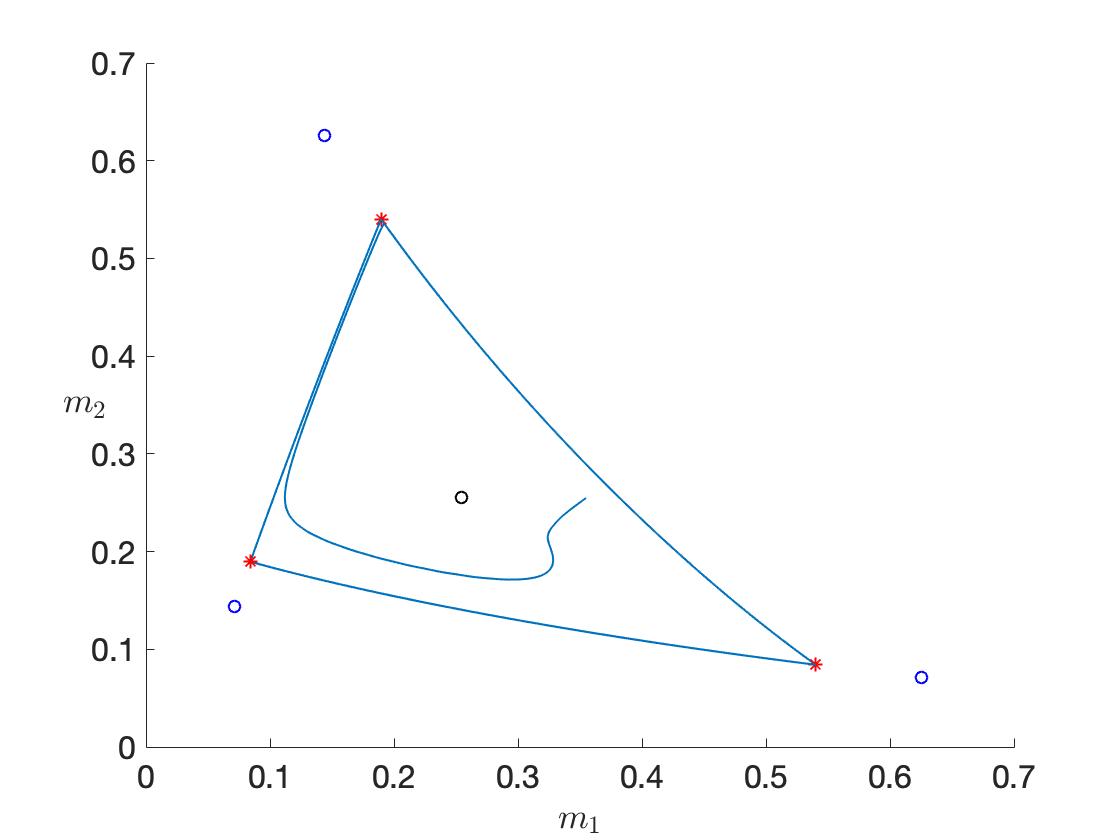}}
   \subfigure[\label{fig:regionb2_subplotb}]{\includegraphics[width=0.8\columnwidth]{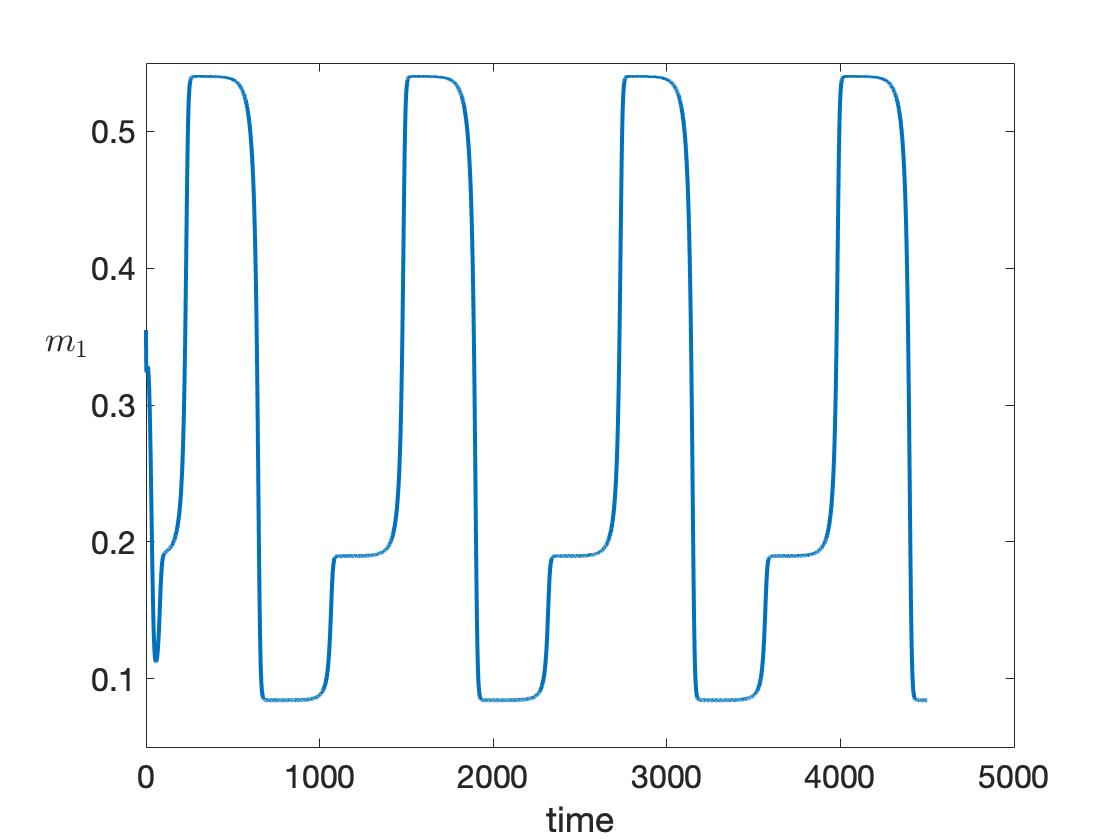} }
  
\caption{(a) Periodic solution appearing for \com{the} parameter values $\alpha = 1.3224$, $\beta = 1.6018$, $\mu= 0.03$, and initial condition \mbox{${\bf m} =( 0.3548,  0.2548, 0.2548)$}. The blue circles represent stable fixed points, while the red stars represent unstable fixed points. (b) Plot of the first component $m_1$  vs. time.}
  \label{fig:regionb2}
\end{figure}

    The heteroclinic connections that exist in the original May--Leonard model~\eqref{e:ML}, and which join the single-population equilibria, live in the invariant coordinate planes $m_i=0$, $i=1,2,3$. In these invariant sets, two single-population fixed points appear as a saddle and sink, and one can prove that their respective unstable and stable manifolds intersect transversely (Schuster et al.~\cite{schuster1979}). In the linearly-perturbed model~\eqref{e:MLperturb}, saddle fixed points only appear in Region B in Fig.~\ref{f:phasediagram-MLRPS}. Since these fixed points correspond to triple-population fixed points, they no longer lie on the coordinate planes. In addition, these planes are no longer invariant sets as soon as $\mu$ becomes positive. Consequently, the heteroclinic connections need to occur in $\R^3$. Since each fixed point has a 2-dimensional stable manifold and 1-dimensional unstable manifold, it then follows that the intersection of these manifolds is no longer robust. \com{As a result, it is unlikely that the heteroclinic cycle persists when $\mu>0$.}

\com{
\begin{Remark}\label{r:periodic}
Interestingly, when the values of $\alpha$ and $\beta$ are near the fold bifurcation, but still inside Region B, we find periodic trajectories that persist under small perturbations of the initial conditions and the parameters (Fig.~\ref{fig:regionb2}). However, when $\alpha$ and $\beta$ are well within Region B, these cyclic solutions are lost and trajectories approach one of the three stable equilibria. We suspect that these periodic solutions are the remnants of the nonperiodic trajectories that approach the heteroclinic cycle in the original May--Leonard model~\eqref{e:ML}. 
\end{Remark}
}

%%%%%%%%%%%%%%%%%%%%%%%%%%%%%%%%%%%%%%%%%%%%%%%%%%%%%%%%%%%%%%%%%%%
\subsection{Region D}\label{ss:regionD}

\com{Region D is the part of the $\alpha$--$\beta$ plane that lies below the Hopf bifurcation line \eqref{e:HopfBifurcationLine} and that is also enclosed by the bifurcation curve labeled 2) in Fig.~\ref{fig:folds}.
We summarize our findings for this region in the following proposition and present numerical evidence for these results.

\begin{prop}
For parameter values $\alpha$ and $\beta$ within Region D, the
linearly-perturbed May--Leonard model \eqref{e:MLperturb} has
a total of eight nonnegative hyperbolic steady states: the trivial-population equilibrium, $e_0$~\eqref{e:fixedpointsML_zero}, which is unstable, the equal-population equilibrium, $e_c$~\eqref{e:fixedpointsML_equalpop}, which is stable,
and an additional six fixed points that emerge through a fold bifurcation, three of which are stable.
\end{prop}

The bifurcation diagram illustrated in Fig.~\ref{fig:folds} justifies the results stated in the above proposition. Below the Hopf line, the equal-population equilibrium, $e_c$~\eqref{e:fixedpointsML_equalpop}, is stable, while above the bifurcation curve labeled 2) in Fig.~\ref{fig:folds}, numerical continuation, together with the results in Appendix A, indicates the emergence of six new steady states.
The stability of these fixed points was tested numerically and a similar plot as shown in Fig. \ref{fig:stabregionb} was obtained (not shown). Thus, we conclude that three of the six fixed points are stable.
}

\com{As an illustration of the dynamics in Region D}, in Fig.~\ref{fig:regiond}, we plot only the positive equilibria and a sample trajectory in the $m_1$--$m_2$ plane for values of $\alpha =1.3$, $\beta =1.3$, and $\mu =0.03$.  The equal-population equilibrium $e_c$~\eqref{e:fixedpointsML_equalpop} is surrounded by the family of three unstable fixed points, while the second set of stable steady states appear in the outskirts of the plot (Fig.~\ref{fig:regiond_subplota}). 
\begin{figure}[t] %  figure placement: here, top, bottom, or page
   \centering
   \subfigure[\label{fig:regiond_subplota}]{\includegraphics[width=0.8\columnwidth]{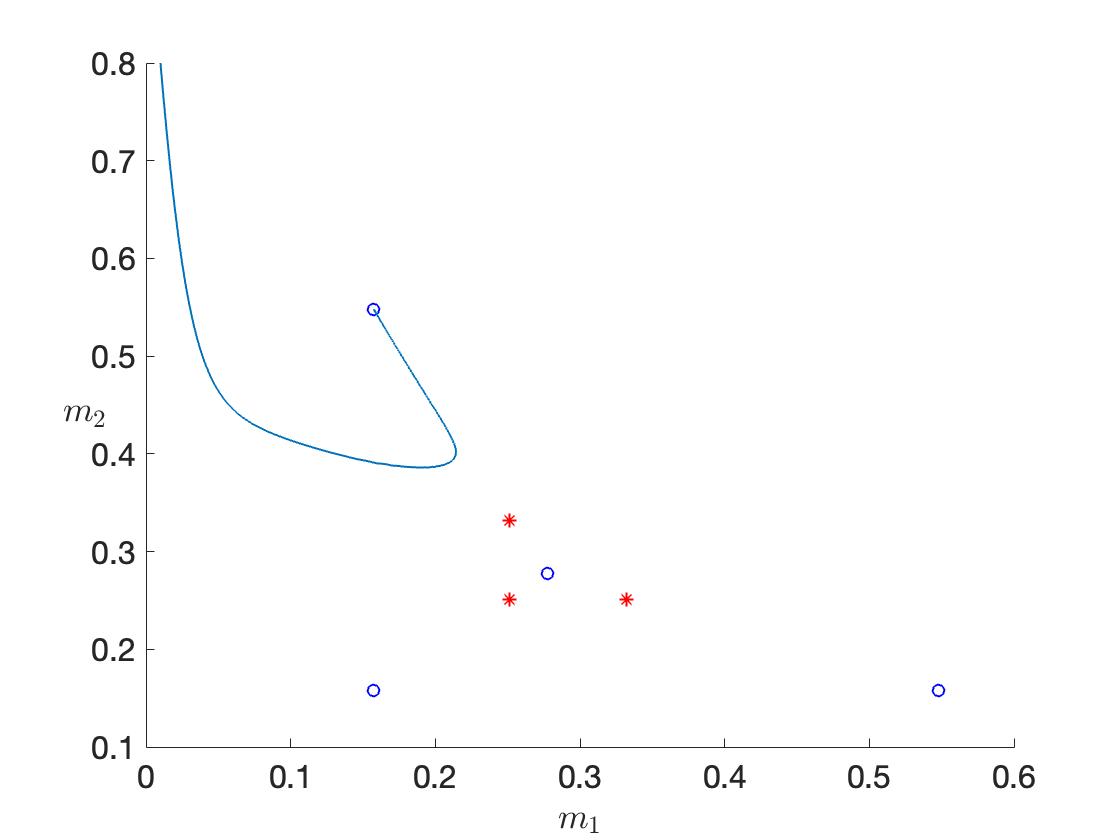} }
   \subfigure[]{ \includegraphics[width=0.8\columnwidth]{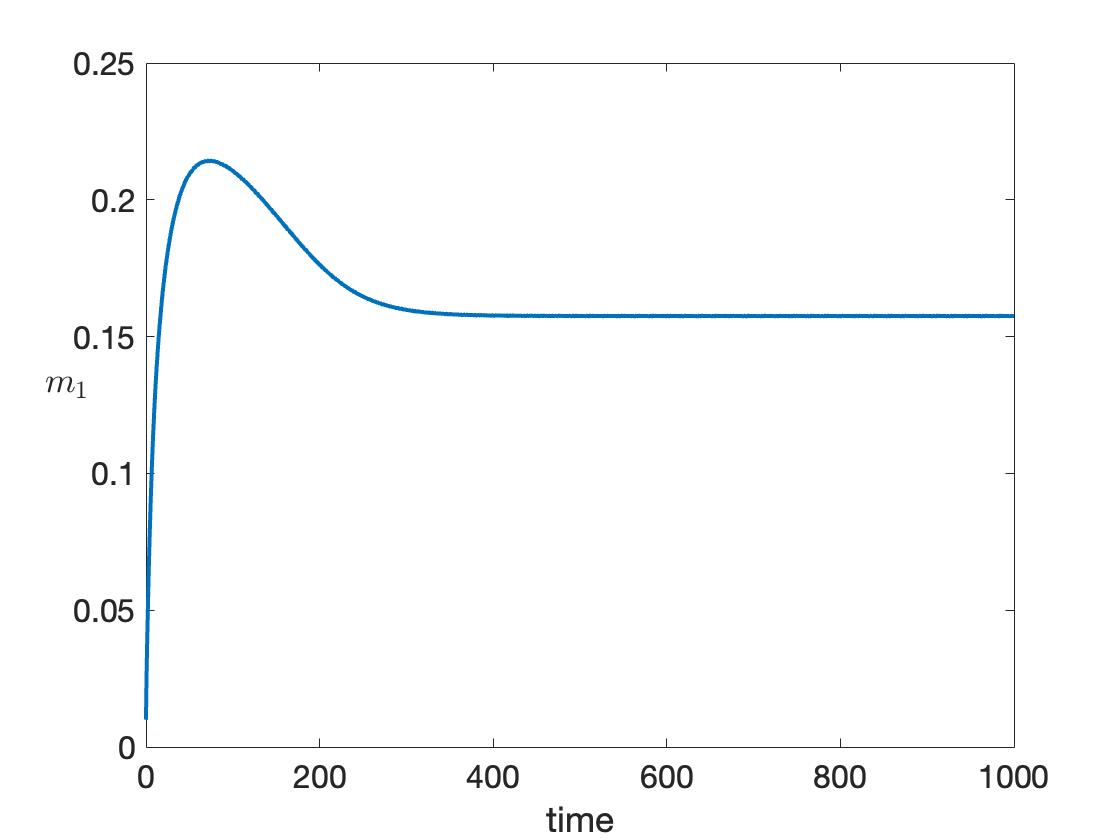} }
   \caption{(a) Sample trajectory for \com{the} parameter values $\alpha =1.3$, $\beta=1.3$, and $\mu= 0.03$, with initial condition ${\bf m} = (0.01,0.8,0.7)$. \com{The} blue circles represent stable fixed points, while the red stars represent unstable fixed points. (b) Plot of the first component $m_1$  vs. time.}
   \label{fig:regiond}
\end{figure}

%\newpage
%%%%%%%%%%%%%%%%%%%%%%%%%%%%%%%%%%%%%%%%%%%%%%%%%%%%%%%%%%%%%%%%%%%
\section{Discussion} \label{sec:discussion}

In this work, we considered the effects of adding a linear perturbation to the three species competition \com{model set forth by May and Leonard}~\cite{may-leonard} and identified changes in the resulting dynamics of the system.
In particular, we focused on linear  perturbations representing global mutations, where each of the three species in the model can mutate into the other two with the same constant rate (see also Tuopo and Strogatz~\cite{toupo-strogatz}). As a result, the perturbed May--Leonard model~\eqref{e:MLperturb} has linear and quadratic terms describing the competition among and mutation between three species.

Not surprisingly, we found that adding a linear term to the equations changed the number and structure of 
the system's fixed points. While both the original and the linearly-perturbed models possess the trivial and equal-population equilibria, the single- and dual- population steady states found in the original May--Leonard system~\eqref{e:ML} are no longer present in our model~\eqref{e:MLperturb}.  Instead,  adding global mutations results in six triple-population fixed points,  some of which have negative components. 
In this paper, we studied \com{the existence and stability of those equilibria which possess only nonnegative components, since they are the only physically relevant fixed points.} This was done using a combination of mathematical analysis and numerically-produced bifurcation diagrams. 

Our results are summarized in a stability diagram (Fig.~\ref{f:phasediagram-MLRPS}), which splits \com{the $\alpha$--$\beta$} parameter space into four distinct regions labeled A, B, C, and D. In Region A, we found that the linearly perturbed system \eqref{e:MLperturb} has only two nonnegative steady states, the \com{trivial-} and equal-population fixed points. We proved that the equal-population equilibrium is the only stable fixed point in this region. We also showed that as one moves from Region A into Region C, the equal-population equilibrium undergoes a supercritical Hopf bifurcation. As a result, the system exhibits periodic solutions, which we showed persist far from the bifurcation line (labeled  C$^\prime$). In Region B, global mutations give rise to six triple-population steady states, all of which have positive components. \com{Three of these equilibria are stable, thus the final state of the system in this region depends on the initial conditions.} 
Finally, in Region D, all eight possible equilibria have nonnegative components. In this region, the equal-population fixed point as well as three of the triple-population steady states are stable. Consequently, the long-term dynamics of the system in Region D also depend on the initial conditions.

The results presented here are specific to the May--Leonard model, which is symmetric with respect to cyclic permutations of the variables, and for linear perturbations that respect this symmetry. We have shown that when this perturbation is small enough, the stability diagram of the modified system~\eqref{e:MLperturb} closely resembles that of the original model~\eqref{e:ML}. In particular, our results show that while the locations of the fixed points 
shift by a small amount 
\com{when small global mutations are added},
their stability is preserved by such a perturbation. Consequently,  the heteroclinic cycle that is present in the original May--Leonard equations~\eqref{e:ML} disappears when
\com{considering the linearly-perturbed system~\eqref{e:MLperturb}.}
This happens not because the perturbation breaks the \com{cyclic symmetry of the system}, but \com{rather} because it dislodges the fixed points from the invariant simplex \com{where the heteroclinic connections take place (i.e., the equilibria no longer lie in the plane $m_1+m_2+m_3 =N$, where $N$ represents the total population).} Similar behavior should be expected for other $n$-dimensional systems, 
provided that the heteroclinic cycle lies in an $(n-1)$-dimensional invariant set and that the added linear perturbation is small and respects the  cyclic symmetry of the system.

As mentioned in the introduction \com{(Section~\ref{sec:introduction})}, the May--Leonard model~\eqref{e:ML} closely resembles the replicator equations used in evolutionary game theory. Indeed, both systems use a cyclic dominance competition pattern, \com{however,} while the May--Leonard model tracks the total population of the different species, replicator games focus on population densities. As a result, when global mutations are introduced,
  the phase diagram of the perturbed May--Leonard model~\eqref{e:MLperturb} (Fig.~\ref{f:phasediagram-MLRPS}) has a much richer structure compared to the phase diagram of the rock--paper--scissors game found in Toupo and Strogatz~\cite{toupo-strogatz}. 
In particular, our results show that in addition to enlarging the region of parameter space where cyclic behavior can be expected, a linear perturbation modeling cyclic mutation also foments coexistence of species: it gives rise to strictly positive steady states that are distinct from the equal-population equilibrium. In biological terms, this would imply that allowing species \com{to switch} from one strategy to another with a small transition, or mutation rate, can favor biodiversity.
\com{Because one can view the linear perturbation as a cooperative force in the modified system~\eqref{e:MLperturb}, our results also corroborate observations that coexistence is stabilized in cooperative systems.}

 Finally, although 
the simplex, $m_1+m_2+m_3 = N$, is no longer an attracting set for the  linearly-perturbed May--Leonard system~\eqref{e:MLperturb}, we strongly suspect that a similar invariant object exists. Indeed, our numerical simulations suggest the presence of a compact and attracting 2-D manifold. However, analytically proving the existence of such a carrying simplex remains an open question. 
Notice, though, that the
existence of an attracting 2-D manifold, together with the Poincar{\'e}--Bendixson theorem, would imply that trajectories of the modified equations~\eqref{e:MLperturb} can only approach a stable fixed point, a periodic orbit, or a heteroclinic cycle. Nonetheless,
it is possible that when extending the equations to the 4-D case, the dynamics of the system become chaotic. 
For example, previous numerical work by Wang and Xiao~\cite{wang-xiao} demonstrates that periodic solutions in the 4-D Lotka--Volterra system 
can undergo successive period-doubling cascades. 
It would be interesting to see if similar chaotic behavior is present in a 4-D linearly-perturbed May--Leonard model. We leave these and related musings as open questions and future work.

%%%%%%

%%%%%%%%%%%%%%%%%%%%%%%%%%%%%%%%%%%%%%%%%%%%%%%%%%%%%%%%%%%%%%%%%%%
\begin{acknowledgments}
G.J. acknowledges support from NSF grant DMS-1911742. T.L.S. acknowledges support from a Simons Collaboration Grant for Mathematicians (\verb|#|710482) and NSF grant DMS-2151566. 

The authors would like to thank the anonymous reviewers for their careful reading and many helpful suggestions.

\end{acknowledgments}

%%%%%%%%%%%%%%%%%%%%%%%%%%%%%%%%%%%%%%%%%%%%%%%%%%%%%%%%%%%%%%%%%%%
\section*{Author Declarations}
\subsection*{Conflict of Interest}
The authors have no conflicts to disclose.
\subsection*{Author Contributions}
All authors contributed equally to this work.
%%%%%%%%%%%%%%%%%%%%%%%%%%%%%%%%%%%%%%%%%%%%%%%%%%%%%%%%%%%%%%%%%%%
\section*{Data Availability}

The data that support the findings of this study are openly available in a GitHub repository at \url{https://github.com/tstepien/linearly-perturbed-May-Leonard}  [v1.0.0]. The source code is platform independent and written in MATLAB and Mathematica.

\appendix

\section{Second-Order Approximation of Fixed Points} \label{sec:appendix}

For small values of $\mu$, we obtain an 
expression for the equilibria of the linearly-perturbed May--Leonard model~\eqref{e:MLperturb} using
a perturbation analysis together with 
the software system Mathematica.

First, we find expansions for the equilibria that bifurcate from the single-population fixed points $e_i$~\eqref{e:fixedpointsML_singlepop}. We work on $e_1$; equilibria corresponding to $e_2$ and $e_3$ can be obtained by permutation because of the symmetry in the system.  Setting $U =(m_1, m_2,m_3)$, we write
\begin{equation*}
    U(\mu) = U_0 + \mu U_1 + \mu^2 U_2 + \mu^3 U_3 + \cdots ,
\end{equation*}
with $U_0 = e_1$~\eqref{e:fixedpointsML_singlepop}. Inserting this ansatz into system  \eqref{e:MLperturb}, one finds that \mbox{$U_1= (u_1,v_1,w_1)$} is given by
\begin{subequations}
\begin{align*}
    v_1 &= \frac{1}{\beta-1} , \\
    w_1 &= \frac{1}{\alpha-1}, \\
    u_1  &= -\alpha v_1 -\beta w_1 -2 = \frac{\alpha}{1-\beta} +\frac{\beta}{1-\alpha} - 2,
\end{align*}
\end{subequations}
while $U_2= (u_2,v_2,w_2)$ has the components
\begin{subequations}
\begin{align*}
   v_2 &= \frac{\alpha}{(\beta-1)^3}+\frac{2\beta-1}{(\alpha-1)(\beta-1)^2}+\frac{1-\alpha\beta}{(\alpha-1)(\beta-1)^3}, \\
    w_2 &= \frac{\beta}{(\alpha-1)^3}+\frac{2\alpha-1}{(\alpha-1)^2(\beta-1)}+\frac{1-\alpha\beta}{(\alpha-1)^3(\beta-1)}, \\
    u_2 & = \frac{\alpha+\beta-2}{(\alpha-1)(\beta-1)} - \alpha v_2 -\beta w_2 \nonumber\\
    & = \frac{1}{1-\alpha}+\frac{\beta(\alpha-\beta)}{(\alpha-1)^3}-\frac{3+\alpha}{(\alpha-1)(\beta-1)}\nonumber \\
    & \qquad +\frac{\alpha}{(\beta-1)^2}-\frac{\alpha(\alpha-1)}{(\beta-1)^3}.
\end{align*}
\end{subequations}

A similar analysis allows us to find expansions for the equilibria that bifurcate from the dual-population fixed points $f_i$~\eqref{e:fixedpointsML_dualpop}. We work on $f_3$, noting that the other equilibria can be obtained by symmetry.
The components of $U_1= (u_1,v_1,w_1)$ are given by
\begin{subequations}
\begin{align*}
    w_1 & =  \frac{2 -\alpha - \beta}{\alpha\beta-1-\alpha(\alpha-1)-\beta(\beta-1)},\\
    u_1 & = \frac{1-3\alpha+\alpha^3+4\alpha\beta-2\alpha^2\beta-\beta^2}{(\alpha-1)(\beta-1)(\alpha\beta-1)}-\frac{\alpha^2-\beta}{\alpha\beta-1}w_1, \\
    v_1 & = \frac{1-\alpha^2-3\beta+4\alpha\beta-2\alpha\beta^2+\beta^3}{(\alpha-1)(\beta-1)(\alpha\beta-1)}+\frac{\alpha-\beta^2}{\alpha\beta-1}w_1, 
\end{align*}
\end{subequations}
and $U_2= (u_2,v_2,w_2)$ is given by
\begingroup
\allowdisplaybreaks
\begin{subequations}
\begin{align*}
    \begin{split}
    w_2 & = \frac{(\alpha\beta-1)(1+\alpha-2\beta)(1-2\alpha+\beta)(-2+\alpha+\beta)}{(\alpha-1)(\beta-1)[\alpha\beta-1-\alpha(\alpha-1)-\beta(\beta-1)]^2},
    \end{split}\\
    \begin{split}
    u_2  =& \frac{(2-\alpha-\beta)(1-3\alpha+4\alpha^2-2\alpha^3+2\alpha^4-2\alpha^5+\alpha^6-3\beta)}{(\alpha-1)^3(\beta-1)^3(\alpha\beta-1-\alpha(\alpha-1)-\beta(\beta-1))}\\
    & + \frac{(2-\alpha-\beta)(8\alpha\beta-12\alpha^2\beta+\alpha^3\beta+4\alpha^4\beta-3\alpha^5\beta)}{(\alpha-1)^3(\beta-1)^3(\alpha\beta-1-\alpha(\alpha-1)-\beta(\beta-1))}\\
    & + \frac{(2-\alpha-\beta)(3\beta^2-4\alpha\beta^2+15\alpha^2\beta^2-8\alpha^3\beta^2+4\alpha^4\beta^2)}{(\alpha-1)^3(\beta-1)^3(\alpha\beta-1-\alpha(\alpha-1)-\beta(\beta-1))}\\
    & + \frac{(2-\alpha-\beta)(-2\beta^3-5\alpha\beta^3-2\alpha^2\beta^3-\alpha^3\beta^3+2\beta^4)}{(\alpha-1)^3(\beta-1)^3(\alpha\beta-1-\alpha(\alpha-1)-\beta(\beta-1))}\\
    & + \frac{(2-\alpha-\beta)(3\alpha\beta^4-\beta^5)}{(\alpha-1)^3(\beta-1)^3(\alpha\beta-1-\alpha(\alpha-1)-\beta(\beta-1))}\\
    & -\frac{\alpha^2-\beta}{\alpha\beta-1} w_2,
    \end{split}\\
    \begin{split}
    v_2  =& \frac{(2-\alpha-\beta)(1-3\alpha+3\alpha^2-2\alpha^3+2\alpha^4-\alpha^5-3\beta)}{(\alpha-1)^3(\beta-1)^3(\alpha\beta-1-\alpha(\alpha-1)-\beta(\beta-1))}\\
    & + \frac{(2-\alpha-\beta)(8\alpha\beta-4\alpha^2\beta-5\alpha^3\beta3\alpha^4\beta+4\beta^2)}{(\alpha-1)^3(\beta-1)^3(\alpha\beta-1-\alpha(\alpha-1)-\beta(\beta-1))}\\
    & + \frac{(2-\alpha-\beta)(-12\alpha\beta^2+15\alpha^2\beta^2-2\alpha^3\beta^2-2\beta^3)}{(\alpha-1)^3(\beta-1)^3(\alpha\beta-1-\alpha(\alpha-1)-\beta(\beta-1))}\\
    & + \frac{(2-\alpha-\beta)(\alpha\beta^3-8\alpha^2\beta^3-\alpha^3\beta^3+2\beta^4+4\alpha\beta^4)}{(\alpha-1)^3(\beta-1)^3(\alpha\beta-1-\alpha(\alpha-1)-\beta(\beta-1))}\\
    & + \frac{(2-\alpha-\beta)(4\alpha^2\beta^4-2\beta^5-3\alpha\beta^5+\beta^6)}{(\alpha-1)^3(\beta-1)^3(\alpha\beta-1-\alpha(\alpha-1)-\beta(\beta-1))}\\
    & +\frac{\alpha-\beta^2}{\alpha\beta-1} w_2.
    \end{split}
\end{align*}
\end{subequations}
\endgroup
%%%%%%%%%%%%%%%%%%%%%%%%%%%%%%%%%%%%%%%%%%%%%%%%%%%%%%%%%%%%%%%%%%%
\section*{References}
\nocite{*}
\bibliography{ML-perturb-bib}% Produces the bibliography via BibTeX.

\end{document}